\documentclass[12pt]{amsart}
\usepackage{amssymb,latexsym,amsmath,amscd,mathrsfs,yfonts,hyperref}
\usepackage{fullpage}
\usepackage{epsfig}
\input xy
\xyoption{all}

\newcommand{\RR}{\mathbb R}
\newcommand{\ZZ}{\mathbb Z}
\newcommand{\QQ}{\mathbb Q}
\newcommand{\CC}{\mathbb C}

\newcommand{\NN}{\mathbb N}

\newcommand{\cpt}{\mathbb K}

\newcommand{\dlim}{{\varinjlim}}
\newcommand{\ilim}{{\varprojlim}}

\def\T{\mathcal T}

\def\loc{\textup{loc}}

\def\cpc{\textup{cpc}}
\def\deg{\textup{deg}}

\def\ch{\textup{ch}}
\def\C{\textup{C}}
\def\biv{\textup{biv}}

\def\KKcat{\mathtt{KK_{C^*}}}

\def\TAlg{\mathrm{T^*Alg}}

\def\ev{\textup{ev}}

\def\RK{\textup{RK}}
\def\id{\mathrm{id}}
\def\tr{\textup{tr}}

\def\hT{\hat{T}}

\def\X{\mathcal X}

\def\Hom{\textup{Hom}}

\def\ker{\textup{ker}}

\def\HL{\textup{HL}}
\def\op{\textup{op}}

\def\hJ{\hat{J}}
\def\Ban{\mathtt{Ban}}

\def\brot{\hat{\otimes}_b}

\def\K{\textup{K}}
\def\fC{\mathfrak C}
\def\skk{\text{$\sigma$-$\textup{kk}$}}

\def\alg{\textup{alg}}
\def\Ho{\textup{Ho}}

\def\HP{\textup{HP}}

\def\indBan{\overset{\longrightarrow}{\mathtt{Ban}}}
\def\cS{\mathcal{S}}

\def\diss{\textup{diss}}

\def\prot{\hat{\otimes}}
\def\colim{\textup{colim}}

\def\cJ{\mathcal J}

\def\CAlg{\mathtt{Alg_{C^*}}}

\def\KK{\textup{KK}}

\def\H{\textup{H}}

\def\pct{\textup{Cpt}}

\def\cC{\mathcal C}

\def\1{\bf{1}}
\def\sC{\text{$\sigma$-$C^*$}}

\def\an{\textup{an}}
\def\HH{\textup{HH}}

\def\knew{\mathfrak{(new)}}

\def\Csep{\mathtt{Sep_{C^*}}}

\def\grot{\hat{\otimes}_\pi}

\def\cU{\mathcal U}

\def\CT{\textup{CT}}
\def\cT{\mathcal T}

\def\sep{\textup{sep}}
\def\cL{\mathcal L}

\def\BKK{\textup{BKK}}

\def\fT{\mathfrak T}
\def\HLcat{\mathtt{HL_{\sC}}}
\def\kh{\mathtt{(kh)}}
\def\ks{\mathtt{(ks)}}
\def\kk{\textup{kk}}
\def\ke{\mathtt{(ke)}}
\def\kt{\mathtt{(kt)}}

\def\od{\textup{od}}
\def\cK{\mathcal K}

\def\A{\mathcal{A}}
\def\sCalg{\mathtt{Sep_{\sC}}}

\def\cB{\mathcal{B}}

\newcommand{\map}{\rightarrow}
\newcommand{\functor}{\longrightarrow}

\def\cT{\mathcal T}

\newcommand{\beq}{\begin{eqnarray}}
\newcommand{\beqn}{\begin{eqnarray*}}
\newcommand{\eeq}{\end{eqnarray}}
\newcommand{\eeqn}{\end{eqnarray*}}

\newtheorem{thm}{Theorem}

\newtheorem{lem}[thm]{Lemma}
\newtheorem{prop}[thm]{Proposition}
\newtheorem{cor}[thm]{Corollary}
\newtheorem{ex}[thm]{Example}
\newtheorem{defn}[thm]{Definition}
\newtheorem{rem}[thm]{Remark}

\begin{document}

\title[Twisted $\K$-theory]{Twisted $\K$-theory, $\K$-homology and bivariant Chern--Connes type character of some infinite dimensional spaces}
\author{Snigdhayan Mahanta}
\email{s.mahanta@uni-muenster.de}
\address{Mathematical Institute,
University of Muenster,
Einsteinstrasse 62,
48149 Muenster, Germany}
\thanks{This research was supported by Australian Research Council's Discovery Projects funding scheme (project number DP0878184), ERC through AdG 267079 and the Deutsche Forschungsgemeinschaft (SFB 878).}

\maketitle

\begin{abstract}
We study the twisted $\K$-theory and $\K$-homology of some infinite dimensional spaces, like $SU(\infty)$, in the bivariant setting. Using a general procedure due to Cuntz we construct a bivariant $\K$-theory on the category of separable $\sC$-algebras that generalizes both twisted $\K$-theory and $\K$-homology of (locally) compact spaces. We construct a bivariant Chern--Connes type character taking values in bivariant local cyclic homology. We analyse the structure of the dual Chern--Connes character from (analytic) $\K$-homology to local cyclic cohomology under some reasonable hypotheses. We also investigate the twisted periodic cyclic homology via locally convex algebras and the local cyclic homology via $C^*$-algebras (in the compact case). 
\end{abstract}

\begin{center}
{\bf Introduction}
\end{center}

Twisted $\K$-theory and cohomology theories have attracted a lot of attention lately due to their importance in string theory. The notion of twisted $\K$-theory can be traced back to a paper by Donovan--Karoubi \cite{DonKar} (see \cite{KarTwist} for a recent survey). The {\em geometric twists} of $\K$-theory are parametrized by the third integral cohomology group. The concept of twisted $\K$-theory was limited to twists with respect to a torsion class until Rosenberg showed how to deal with arbitrary (possibly non-torsion) classes using continuous trace $C^*$-algebras \cite{RosCT} (see also \cite{ASTwistedK1,ASTwistedK2}). One advantage of the operator theoretic viewpoint of twisted $\K$-theory is that it enables us to use all the tools from noncommutative geometry \cite{ConBook}, although there are various other elegant approaches, viz, homotopy theory viewpoint \cite{ABG,FHT}, bundle gerbe viewpoint \cite{BCMMS}, groupoid and stack viewpoint \cite{TuXuLG,DanTDuality}, and so on. Twisted $\K$-homology is an equally important theory (see, for instance, \cite{EvaGan,MeiTwist}) and the pairing between the two theories has applications to topological $\textup{T}$-duality \cite{BMRS2,KQ}.

The aim of this article is to provide a formalism for both twisted $\K$-theory and $\K$-homology as a bivariant theory with composition product on a certain category of noncommutative spaces that are not necessarily compact (or even locally compact). The appropriate function algebra for the noncommutative analogue of a compactly generated and completely Hausdorff space is a pro $C^*$-algebra \cite{NCP1,NCP2}. It turns out that there is a reasonable bivariant $\K$-theory on the category of pro $C^*$-algebras due to Weidner \cite{Weidner}. Unfortunately Weidner's bivariant $\K$-theory lacks the desirable universal characterisation that enables us to construct interesting bivariant natural transformations into other theories very easily. Therefore, we adopt a different strategy suggested by Cuntz \cite{CunGenBivK} to construct a bivariant $\K$-theory, denoted by $\skk$-theory, on the category of separable $\sC$-algebras, which is a full subcategory of pro $C^*$-algebras.  Using some results of Bonkat \cite{
Bonkat} we eventually show that our $\skk$-theory agrees with Weidner's bivariant $\K$-theory on the subcategory of separable and nuclear $\sC$-algebras. It remains an open question whether they agree on all separable $\sC$-algebras. Using the universal characterisation of $\skk$-theory we also construct a bivariant Chern--Connes type character on the category of separable $\sC$-algebras taking values in bivariant local cyclic homology \cite{Puschnigg,PuschniggHL}. In \cite{MatSte} the authors constructed a twisted Chern--Connes character from the twisted $\K$-theory to the twisted periodic cyclic homology using techniques from noncommutative geometry. They further showed that this map agrees with the original geometric construction of the twisted Chern character and becomes an isomorphism after tensoring with $\CC$. Their theory was restricted to compact spaces and it was subsequently extended to orbifolds in \cite{TuXuChern}. Our bivariant Chern--Connes type character generalizes the twisted Chern--Connes 
character 
of \cite{MatSte} but in general it does not become an isomorphism after tensoring with $\CC$. This phenomenon can be explained by the infinite dimensionality of the underlying spaces on which the theory is being considered. We mostly consider those infinite dimensional spaces, whose underlying topological spaces are paracompact, Hausdorff and countably compactly generated, i.e., countable direct limits of compact spaces. Let us comment further on the level of generality of the infinite dimensional spaces that can be studied via $\sC$-algebras.

In many interesting applications of twisted $\K$-theory to physics and topology one encounters large spaces, e.g., $U(\infty)$, $SU(\infty)$, $\CC P^\infty$, etc. The Bott unitary group $U(\infty)$ is a classifying space for topological $\K^1$-theory and hence an interesting object in topology. In the Wess--Zumino--Witten models one typically takes a Lie group, like $SU(n)$, as the target space. It is plausible that allowing the target space to be an infinite dimensional Lie group, like $SU(\infty)$, is also interesting. Such objects appear in the `large $N$ limit' discussions in string theory. Specifically, Gopakumar--Vafa advocate the study of $SU(\infty)$ gauge theories in \cite{GopaVafa} and (infinite) matrix models were shown to be relevant to $M$-theory \cite{BFSS}. The category of separable $\sC$-algebras is sufficient to treat all the above-mentioned examples, i.e., $$U(\infty) = \dlim_n U(n), \hspace{2mm} SU(\infty)=\dlim_n SU(n), \hspace{2mm} \CC P^\infty=\dlim_n\CC P(n).$$ Furthermore, for a compact Hausdorff 
topological group $G$, one may choose a countably compactly generated and Hausdorff model for $EG$. We concentrate on the example of $SU(\infty)$ throughout, which is an infinite dimensional Lie group, and known as the {\em universal gauge group} in the physics literature \cite{HarMoo,CarMic}. This group was quantized in the setting of $\sC$-quantum groups and its (representable) $\K$-theory was computed in \cite{Qhom}. Another important class of infinite dimensional spaces is that of loop groups \cite{PreSeg} (see also \cite{HekMic}). Although such objects cannot be studied via $\sC$-algebras, we argue in the first few sections of the paper (see Section \ref{locConK} and Section \ref{HPT}) that there is a completely satisfactory theory using locally convex algebras \cite{CunWeyl} if one is merely interested in twisted $\K$-theory and periodic cyclic homology computations (and not their full bivariant generalizations). In a subsequent paper we are going to describe another bivariant theory that is capable of 
handling all such infinite dimensional spaces.

We demonstrate within the text that twisted $\K$-theory and $\K$-homology are fairly computable theories. Connes originally constructed cyclic cohomology (and its periodic version) as a receptacle for the $\K$-homological Chern--Connes character \cite{ConDiff}. It is anticipated that the bivariant Chern--Connes type character would produce yet another computational aid modulo torsion. Thanks to the extensive work on the homology theory of {\em topological algebras} (see, for instance, \cite{Helemskii,MeyCycHom}) there are now several computational strategies available for cyclic homology theory (and its various derivatives). The article is organized as follows:

In Section \ref{locConK} we investigate bivariant $\K$-theory in the context of locally convex algebras. If $X$ is a countably compactly generated and Hausdorff topological space then $\C(X)$ is $\sC$-algebra. From such a space and a twisting datum in terms of a principal projective unitary bundle on it we show how to construct a $\sC$-algebra. Such algebras and their representable $\K$-theory should provide the natural extension of twisted $\K$-theory to countably compactly generated and Hausdorff spaces. However, with some foresight we subsume the discussion in the context of bivariant  $\K$-theory for locally convex algebras developed by Cuntz \cite{CunWeyl}. The reason is that we would also like to investigate the twisted periodic cyclic homology of such spaces and the associated Chern--Connes character map. Using Karoubi density type results one can extract a smooth dense locally convex subalgebra of the aforementioned $\sC$-algebra with the same $\K$-theory. This bivariant generalization will also be 
useful for future applications. We show that for nontrivial twists the twisted $\K$-theory groups of $SU(\infty)$ are trivial (Example \ref{SUK}); however, it is easy to see that there are examples of spaces for which the twisted $\K$-theory groups are nontrivial (see, for instance, Example \ref{NTriv}). Let us emphasize that the discussions in this section are all applicable to very general spaces like quasitopological spaces in the sense of Spanier \cite{Spanier}.

In Section \ref{HPT} we study the twisted periodic cyclic homology theory via locally convex algebras. In this context one also needs to keep track of a smooth structure. Therefore, we show how to construct {\em smooth subalgebras} of $\sC$-algebras (whenever it makes sense). We explain the general construction of the periodic cyclic homology groups of locally convex algebras using the $\X$-complex formalism. Although, we mostly look at the periodic cyclic homology groups, it is presumably a better idea to directly work with the $\X$-complex. Once again we point out that the discussions in this section are all applicable to quasitopological spaces.

In Section \ref{HPCC} we first discuss the periodic cyclic homology valued Chern--Connes character and give an example where the map does not become an isomorphism after tensoring with the complex numbers. As argued before, this is one of the first instances where the infinite dimensional setting differs from the finite dimensional compact setting. We also study the local cyclic homology valued Chern--Connes character in the compact case. In order to do so we prove a Karoubi density type result in local cyclic (co)homology (see Theorem \ref{HLcpt}). We also show that the local cyclic homology valued Chern--Connes character becomes an isomorphism after tensoring with the complex numbers (in the compact case). Hence it can act as a replacement for the periodic cyclic homology as a target for the $\K$-theoretic Chern--Connes character.

The author is not aware of any Karoubi density type result in analytic $\K$-homology. In fact, the bivariant $\K$-theory discussed in Section \ref{locConK} does not produce the correct $\K$-homology, i.e., it does not agree with Kasparov's analytic $\K$-homology on the category of separable $C^*$-algebras. Therefore, in Section \ref{KHom} we use a different bivariant $\K$-theory, denoted by $\skk$-theory (see Definition \ref{newKK}), on the category of separable $\sC$-algebras \cite{CunGenBivK}). This theory agrees with Kasparov's bivariant $\K$-theory, when restricted to the subcategory of separable $C^*$-algebras. We also show that it agrees with the bivariant $\K$-theory developed by Weidner \cite{Weidner} {\em on the subcategory of nuclear and separable $\sC$-algebras}. Using the established properties of Weidner's bivariant $\K$-theory, we prove some useful results about the $\K$-theory and $\K$-homology of nuclear and separable $\sC$-algebras (see Corollary \ref{Nuc}). The $\sC$-algebras that arise as 
twisted noncommutative spaces as described above fall in this category.

For the $\K$-homological Chern character in commutative geometry we refer the readers to \cite{BauDou1}. Due to the lack of Karoubi density in $\K$-homology, the natural codomain for the dual Chern--Connes character from $\K$-homology of separable $\sC$-algebras is local cyclic cohomology \cite{Puschnigg,PuschniggHL}. In Section \ref{HL} we explain how the theory works for separable $\sC$-algebras. In order to do so we describe a functorial construction of an ind-Banach algebra from a $\sC$-algebra. There is an elegant treatment of local cyclic (co)homology on the category of ind-Banach algebras due to Meyer \cite{MeyCycHom}. The passage from $\sC$-algebras to ind-Banach algebras goes through the category of bornological algebras. Therefore, we include a brief discussion of bornological algebras. If one wants to use the $\X$-complex formalism to define local cyclic (co)homology, which is what we do here, then the bornological machinery is indispensable; indeed, there is a bornological completion of the 
algebraic $\X$-complex that produces the desired result.

In Section \ref{HLCC} we construct a natural bivariant Chern--Connes type character, which is the main mathematical result of this paper, and then specialize to the dual local cyclic cohomology valued Chern--Connes character from $\K$-homology.

\begin{thm} (see Theorem \ref{BivCC} below)
There is a natural multiplicative bivariant Chern--Connes type character $\ch^\biv_*:\skk_*(A,B)\cong\skk_*(\cpt\prot A,\cpt\prot B)\map\HL_*(\cpt\prot A,\cpt\prot B)$ on the category of separable $\sC$-algebras.
\end{thm} This multiplicative bivariant Chern--Connes type character is much more powerful as a tool than the disembodied univariant $\K$-theoretic Chern--Connes character and the $\K$-homological dual Chern--Connes character. On the category of {\em locally multiplicative bornological algebras} an $\HL_*$-valued bivariant Chern--Connes character was constructed in \cite{Voigt}. However, the dissection functor on the category of $\sC$-algebras does not land inside the category of locally multiplicative bornological algebras, whence the results of ibid. are not applicable to $\sC$-algebras. Under some mild hypotheses, we exhibit a factorization of the dual Chern--Connes character (see Theorem \ref{dualCC}). We end by discussing a general class of examples (invoking Poincar{\'e} duality type isomorphisms), where some simplifications occur.

\vspace{3mm}
\noindent
{\bf Conventions:} Throughout this article we work over a subcategory of $\CC$-algebras with $\CC$-algebra homomorphisms (typically locally convex algebras or $\sC$-algebras), although many discussions make sense much more generally. Unless otherwise stated, all spaces are assumed to be (completely) Hausdorff and topological or bornological vector spaces are assumed to be complete. Adding a separability condition on a commutative $C^*$-algebra amounts to imposing a metrizability condition on the corresponding space.

\vspace{3mm}
\noindent{\bf Terminology:}
The periodic cyclic homology valued $\K$-theoretic Chern--Connes character is also known as the Connes--Karoubi character in the literature due to Karoubi's contribution to this topic \cite{Kar}. In order to avoid confusion, we stick to the original terminology in the bivariant situation developed by Cuntz.

\vspace{3mm}
\noindent{\bf Acknowledgements:}
The author is extremely grateful to J. Cuntz and R. Meyer for some invaluable email correspondences regarding bivariant $\K$-theory and local cyclic (co)homology respectively. The author is indebted to N. C. Phillips for patiently explaining several technical issues about pro $C^*$-algebras. The author also wishes to thank R. Gopakumar, P. Hekmati and V. Mathai for some discussions regarding twisted $\K$-theory. The author would also like to thank M. Khalkhali, M. Marcolli and J. Mickelsson for their helpful comments. Finally the author expresses his gratitude towards the anonymous referee(s) for pointing out a gap in an earlier version of the paper and suggesting several improvements.

\section{Twisted $\K$-theory via locally convex algebras} \label{locConK}

\subsection{Preliminaries on $\sC$-algebras}
Let $\TAlg$ denote the category of noncommutative topological $*$-algebras with continuous $*$-homomorphisms \cite{Mallios}. Now we recall some basic facts about inverse limit $C^*$-algebras or pro $C^*$-algebas following Phillips \cite{NCP1,NCP2}. A {\em pro $C^*$-algebra} is an object of $\TAlg$, which is complete Hausdorff and whose topology is determined by it continuous $C^*$-seminorms, i.e., a net $\{a_\lambda\}$ converges to $0$ if and only if $p(a_\lambda)$ converges to $0$ in $\RR$ for every $C^*$-seminorm $p$ on $A$. Clearly the category of $C^*$-algebras, denoted $\CAlg$, with $*$-homomorphisms (automatically continuous) is a full subcategory of $\TAlg$. Let $I$ be a small filtered index category. A contravariant functor $I^\op\functor\CAlg$ ($i\mapsto A_i$) produces an inverse system of $C^*$-algebras with $*$-homomorphisms. The inverse limit of this system $\ilim_{i\in I}  A_i$ inside the category $\TAlg$ is a pro $C^*$-algebra. In fact, an object $A\in\TAlg$ is a pro $C^*$-algebra if and only if it arises as a limit of an inverse system of $C^*$-algebras as described above.

Let $A$ be a pro $C^*$-algebra and let $S(A)$ denote the set of all continuous $C^*$-seminorms on $A$. For any $p\in S(A)$ set $\ker(p)= \{a\in A | p(a)=0\}$. Then $\ker(p)$ is a two-sided closed $*$-ideal in $A$ and $A_p= A/\ker(p)$ is actually a $C^*$-algebra. The set $S(A)$ is directed by declaring $p\leqslant q$ if and only if $p(a)\leqslant q(a)$ for all $a\in A$. This converts $\{p\mapsto A_p\}$ into an inverse system and, in fact, $A\cong \ilim_{p\in S(A)} A_p$. Clearly, the limit does not change by passing onto a cofinal subset of $S(A)$. A pro $C^*$-algebra $A$ is called a {\em $\sC$-algebra} if there is a countable (cofinal) subset $S'\subset S(A)$, which determines that topology of $A$; in other words, $A\cong \ilim_{p\in {S'}} A_p$. A space $X=\dlim_{n\in\NN} X_n$ is a {\em countably compactly generated} space if each $X_n$ is compact and Hausdorff, i.e., $X$ is a countable direct limit of compact Hausdorff spaces in the category of topological spaces and continuous maps. The direct limit topology on $X$ is automatically Hausdorff. The category of commutative and unital $\sC$-algebras is (contravariantly) equivalent to the category of countably compactly generated and Hausdorff spaces via the functor $X\mapsto \C(X)$. It is also known that commutative and unital pro $C^*$-algebras model all {\em quasitopological} and {\em completely Hausdorff} spaces. However, the technical aspects of pro $C^*$-algebras are rather cumbersome. For instance, a generic algebraic $*$-homomorphism need not be continuous; even if it is continuous it need not 
have closed range. Fortunately, it is known that any $*$-homomorphism between two $\sC$-algebras is automatically continuous.

Given any $\sC$-algebra $A$ one can find a confinal subset $S'\subset S(A)$, such that $S'\simeq \NN$ (as a linearly directed set). Therefore, one can explicitly write $A$ as a countable inverse limit of $C^*$-algebas, $A\cong \ilim_{n\in\NN} A_n$. Furthermore, the connecting $*$-homomorphisms of the inverse system $\{A_n\map A_{n-1}\}_{n\in\NN}$ can be arranged to be surjective without altering the inverse limit $\sC$-algebra $A\cong \ilim_n A_n$. This is done by replacing each $A_n$ by the intersection of all the images of the homomorphisms $A_m\map A_n$, $m\geqslant n$. In the sequel, we shall freely use this explicit presentation of a $\sC$-algebra as an inverse limit of a countable inverse system of $C^*$-algebras with surjective $*$-homomorphisms.

\begin{rem}
The category $\CAlg$ also contains all small inverse (filtered) limits. In general, a pro $C^*$-algebra is not a $C^*$-algebra, since the limit is taken in the larger category $\TAlg$. For instance, let $X=\colim_{i\in\NN} X_i$ be a countably compactly generated space with each $X_i$ compact. Here $X$ is endowed with the direct limit topology, which is automatically Hausdorff. Then $\{\C(X_i)\}$ forms naturally a countable inverse system of $C^*$-algebras with restriction $*$-homomorphisms. The limit of this system in $\TAlg$ is $\C(X)$, whereas the limit inside $\CAlg$ is $\C_b(X)$, the unital $C^*$-algebra of norm bounded functions on $X$.
\end{rem}

\subsection{$\K$-theory of locally convex algebras} \label{CKTh}

Phillips defined a $\K$-theory for $\sC$-algebras \cite{PhiRepK}, which is called representable $\K$-theory and denoted by $\RK$. The $\RK$-theory agrees with the usual $\K$-theory of $C^*$-algebras if the input is a $C^*$-algebra and many of the nice properties that $\K$-theory satisfies generalise to $\RK$-theory. We are going to work with a more general $\K$-theory (for locally convex algebras) than $\RK$-theory, which agrees with the latter on the category of $\sC$-algebras. Nevertheless, let us briefly recall from ibid. some of the basic facts about $\sC$-algebras and $\RK$-theory. If $A=\ilim_n A_n$ is a $\sC$-algebra then the {\em stabilization} $A\prot\cpt$ is defined to be $\ilim_n A_n\prot \cpt$, where $\cpt$ denotes the $C^*$-algebra of compact operators and $\prot$ denotes the maximal $C^*$-tensor product.

\begin{enumerate} \label{RK}

\item For each $i\in\NN$, $\RK_i$ is a homotopy invariant abelian group valued functor on the category of $\sC$-algebras.

\item Bott periodicity holds, so that $\RK_i(A)\cong\RK_{i+2}(A)$.

\item \label{agreement} If $A$ is a $C^*$-algebra then there is a natural isomorphism $\RK_i(A)\cong\K_i(A)$.

\item There is a natural isomorphism $\RK_i(A)\cong \RK_i(A\prot\cpt)$, induced by the corner embedding $A\map A\prot\cpt$.

\item Countable products are preserved, i.e., there is a natural isomorphism $$\RK_i(\prod_n A_n)\cong\prod_n\RK_i(A_n).$$\label{product}

\item If $\{A_n\}_{n\in\NN}$ is a countable inverse system of $\sC$-algebras with surjective homomorphisms (which can always be arranged), then the inverse limit exists as a $\sC$-algebra and there is a Milnor $\ilim^1$-sequence

\beqn
0\map {\ilim}^1_n \RK_{1-i}(A_n)\map \RK_i(\ilim_n  A_n)\map \ilim_n \RK_i(A_n)\map 0.
\eeqn

\item If $0\map A\map B\map C\map 0$ is an exact sequence of $\sC$-algebras then there is a $6$-term exact sequence

\beqn
\xymatrix{
\RK_0(A)\ar[r] &\RK_0(B)\ar[r] & \RK_0(C)\ar[d]\\
\RK_1(C)\ar[u] & \ar[l] \RK_1(B) & \ar[l]\RK_1(A).
}
\eeqn
\end{enumerate}

As stated earlier, for our purposes we need a more general $\K$-theory developed by Cuntz \cite{CunBivCh,CunWeyl}, which is defined on the larger category of locally convex algebras with continuous homomorphisms. The category of locally convex algebras is general enough to subsume almost all cases of interest in noncommutative geometry (in the operator algebraic framework). In particular, it contains $\sC$-algebras and the $\K$-theory that we are going to describe is a generalization of the representable $\K$-theory of $\sC$-algebras. All the necessary details of the contents of this section can be found in ibid..

Let $\grot$ the completed projective tensor product between two complete locally convex spaces \cite{GroPro}. By a {\em locally convex algebra} we mean a $\CC$-algebra, whose underlying linear space is a complete locally convex space, such that the algebra multiplication is jointly continuous, i.e., it extends to a continuous linear map $A\grot A\map A$. These are examples of pro-algebra objects in the symmetric monoidal additive category of Banach spaces (monoidal structure given by $\grot$). Let $\CC[0,1]$ denote the algebra of $\CC$-valued $\C^\infty$-functions on $[0,1]$ all whose derivatives vanish at the endpoints $\{0,1\}$. Let $\CC(0,1)$ be the subalgebra of $\CC[0,1]$, which vanish on $\{0,1\}$. For any locally convex algebra $A$, set $A[0,1]^k=\A\grot\CC[0,1]^{\grot k}$ and $A(0,1)^k= A\grot\CC(0,1)^{\grot k}$. The crucial property satisfied by $\K$-theory for locally convex algebras is {\em diffotopy invariance}. Two continuous homomorphisms $\alpha_1,\alpha_2: A\map B$ between locally convex 
algebras are said to be {\em diffotopic} if there is a continuous homomorphism $\phi: A\map B[0,1]$, such that $\alpha_1 =\ev_0\phi$ and $\alpha_2=\ev_1\phi$, i.e., the following diagram consisting of continuous homomorphisms commutes

\beqn
 \xymatrix{
 & B \\
A
\ar[ur]^{\alpha_1}
\ar[r]^{\phi}
\ar[dr]_{\alpha_2}
& B[0,1]
\ar[u]_{\ev_0}
\ar[d]^{\ev_1} \\
& B, } \\
\eeqn where it should be noted that $B[0,1]$ consists of smooth $B$-valued functions on $[0,1]$ all whose derivatives vanish at the endpoints, whence the name diffotopy (as opposed to homotopy).

Let $V$ be a complete locally convex space. Let $T^\alg V= \oplus_{n=1}^\infty V^{\otimes n}$ denote the free algebraic (nonunital) tensor algebra over $V$. There is a canonical linear map $\sigma: V\map T^\alg V$, sending $V$ to the first summand. Equip $T^\alg V$ with the locally convex topology given by the family of all seminorms of the form $p\circ\phi$, where $\phi$ is any homomorphism $T^\alg V\map B$ ($B$ any locally convex algebra and $p$ a continuous seminorm on $B$), such that $\phi\circ\sigma: V\map B$ is a continuous linear map. For any locally convex algebra $A$, let us denote the completed locally convex tensor algebra by $TA$. Using the construction $A\mapsto TA$ one can define a $\K$-theory, but it is not clear whether this theory agrees with the representable $\K$-theory of $\sC$-algebras. To this end, one needs the $m$-algebra modification.

Following \cite{CunBivCh} let us call a locally convex algebra, which is an inverse limit of Banach algebras, an {\em $m$-algebra}. There is a different free (nonunital) tensor algebra construction, which is suitable in the category of $m$-algebras. Denote by $\hT V$ the completion of $T^\alg V$ with respect to the family $\{\hat{p}\,|\, \text{$p$ continuous seminorm on $V$}\}$ of submultiplicative seminorms, where $\hat{p}=\oplus_{n=1}^\infty p^{\otimes n}$. For any $m$-algebra $A$, the tensor algebra $\hT A$ is also an $m$-algebra, and the canonical homomorphism $\pi:T^\alg A\map A$ sending $a_1\otimes\cdots\otimes a_n\mapsto a_1\cdots a_n$ extends to a continuous homomorphism $\pi: \hT A\map A$, which is evidently surjective. Let $\hJ A$ denote the kernel of $\pi$, so that $$0\map \hJ A\map \hT A\overset{\pi}{\map} A\map 0$$ is an extension diagram of $m$-algebras. The constructions $A\mapsto \hJ A$ and $A\mapsto \hT A$ are both functorial with respect to continuous homomorphisms between $m$-algebras.

One needs to extend the construction of $A\mapsto \hT A$ to arbitrary locally convex algebras. For an arbitrary locally convex algebra $A$, the free (nonunital) tensor algebra $\hT A$ is uniquely (up to an isomorphism) characterised by the following properties:

\begin{itemize}
\item If $A$ is an $m$-algebra, then $\hT A$ is the free tensor $m$-algebra described above.
\item The construction $A\mapsto \hT A$ is functorial with respect to continuous homomorphisms.
\item For every $m$-algebra $C$, the natural map $T(A\otimes C)\map TA\grot C$ extends to a continuous map $\hT(A\grot C)\map \hT A\grot C$.
\end{itemize}

We refer the readers to \cite{CunWeyl} for the explicit construction of $\hT A$ for an arbitrary locally convex algebra $A$. Define $\hJ A$ as before and set $\hJ^l A=\hJ^{l-1}(\hJ A)$. Let $\cK$ denote the algebra of smooth compact operators (see section 2.2. of ibid.) and $\langle \, ,\rangle$ denote the set of diffotopy classes of continuous homomorphisms between locally convex algebras. Then for each $k\in\NN$ there is a canonical classifying map $$\langle \hJ^k A,\cK\grot B(0,1)^k\rangle\map\langle \hJ^{k+1} A,\cK\grot B(0,1)^{k+1}\rangle,$$ mapping the diffotopy class of $\alpha:\hJ^k A\map\cK\grot B(0,1)^k$ to that of $\alpha':\hJ^{k+1} A\map\cK\grot B(0,1)^{k+1}$. Here $\alpha'$ is defined as the following classifying map of extensions:

\beqn
\xymatrix{
0\ar[r] & \hJ^{k+1} A \ar[r] \ar[d]^{\alpha'} & \hT \hJ^k A \ar[r] \ar[d] & \hJ^k A \ar[r] \ar[d]^{\alpha} & 0\\
0\ar[r] & B'(0,1)^{k+1} \ar[r] & B'(0,1)^k[0,1)\ar[r] & B'(0,1)^k \ar[r] & 0,
}
\eeqn where $B' =\cK\grot B$ and the bottom sequence is the standard suspension-cone extension of $B'(0,1)^k$.

\begin{defn}[Cuntz, Definition 15.4. of \cite{CunWeyl}]
Given any two locally convex algebras $A,B$ one defines the bivariant $\K$-theory groups as

\beq
\kk_n(A,B) &=\begin{cases} 
\dlim_k \langle \hJ^{k}(J^{n} A),\cK\grot B(0,1)^k\rangle & \text{if  } n\in\NN, \\
\dlim_k \langle \hJ^{k}(A),\cK\grot B(0,1)^{k-n}\rangle & \text{if  } n\in\ZZ\setminus\NN.
\end{cases}
\eeq
\end{defn}

\noindent
Now we define the $\K$-theory of a locally convex algebra $A$ as $\K_n(A)=\kk_n(\CC,A)$. These bivariant $\K$-theory groups agree with those defined earlier in \cite{CunBivCh}, when restricted to the category of $m$-algebras.

\begin{rem}
One can also define the dual $\K$-homology theory by setting $\K^n(A)=\kk_n(A,\CC)$. However, these $\K$-homology groups will not agree with Kasparov's (analytic) $\K$-homology groups, when restricted to the category of separable $C^*$-algebras. We rectify this problem in Section \ref{KHom} by introducing a different version of bivariant $\K$-theory for separable $\sC$-algebras.
\end{rem}

For each $n\in\NN$, the association $A\mapsto\K_n(A)$ defines a covariant diffotopy invariant abelian group valued functor on the category of locally convex algebras. Phillips generalized the $\RK$-theory for $\sC$-algebras in \cite{PhiFreK} to a $\K$-theory for all Fr{\'e}chet algebras, which can be written as a countable inverse limit of Banach algebras (such that all the connecting homomorphisms and the canonical projections have dense ranges). Let us denote this $\K$-theory of Phillips by $\K^P$-theory.

\begin{lem} \label{KP}
Restricted to the category of $\sC$-algebras, the functors $\K_*(-)$ satisfy all the formal properties of $\RK$-theory mentioned above \ref{RK}.
\end{lem}

\begin{proof}
The assertion follows from Theorem 20.2 of \cite{CunBivCh}, which says that there is a natural isomorphism $\K^P_*(A)\cong\K_*(A)$ if $A$ is a Fr{\'e}chet algebra as described above.
\end{proof}

\subsection{Twisted $\K$-theory via continuous trace algebras} \label{CT}
In the sequel we denote by $PU$ the projective unitary group of a separable Hilbert space. If $X$ is a compact space and $P$ is a principal $PU$-bundle on $X$, then one can define a stable continuous trace $C^*$-algebra $\CT(X,P)$, whose topological $\K$-theory is defined to be the twisted $\K$-theory \cite{RosCT}. A good reference for the general theory of continuous trace $C^*$-algebras is \cite{RaeWil}. We are going to generalize this construction to the case where $X$ is a countably compactly generated and paracompact space.

\begin{rem}
It follows from Kuiper's Theorem that $BPU$ has the homotopy type of $K(\ZZ,3)$, so that $\H^3(X,\ZZ)=[X,BPU]$. We need the assumption of paracompactness on $X=\dlim_n X_n$ in order to assert that isomorphism classes of principal $PU$-bundles on $X$ are in bijection with homotopy classes of maps $X\map BPU$. Every compact smooth manifold is paracompact and so is a countable direct limit of such manifolds (see Proposition 3.6 of \cite{GloInfDimLie}).
\end{rem}

Consider a category whose objects are pairs $(X,P)$, where $X$ is a countably compactly generated and paracompact space and $P$ is a fixed choice of principal $PU$-bundle on $X$. A morphism $(X,P)\map (X',P')$ in this category is a continuous map $f:X\map X'$, such that there is a specified isomorphism $f^*(P') \overset{\sim}{\map} P$. Let $\CT(X,P)$ denote the $*$-algebra of all continuous sections of the associated bundle $P_\cpt:= P\times_{PU}\cpt\map X$, where $\cpt$ denotes the $C^*$-algebra of compact operators.

The direct limit does not change if one passes to a cofinal subsystem. Therefore, by passing to a cofinal subsystem we may assume that $X=\cup_i X_i$, where each $X_i$ is compact and $X_0\overset{\iota_0}{\hookrightarrow}X_1\overset{\iota_1}{\hookrightarrow} X_2\cdots$ is a countable system of continuous inclusions. One can restrict the principal $PU$-bundle $P$ via the canonical continuous inclusions $X_i\map X$, which we denote by $P_i$. Each $\CT(X_i,P_i)$ is a stable continuous trace $C^*$-algebra and, in fact, $\{(\CT(X_i,P_i),\iota_i^*)\}$ forms a countable inverse system of $C^*$-algebras and $*$-homomorphisms. The canonical inclusions $X_i\map X$ induce $*$-homomorphisms $\CT(X,P)\map\CT(X_i,P_i)$, which assemble to produce a $*$-homomorphism $\CT(X,P)\overset{\gamma}{\map} \ilim_i \CT(X_i,P_i)$.

\begin{lem}
Let $X$ be a countably compactly generated and paracompact space and $P$ be a principal $PU$-bundle on $X$. Then $\CT(X,P)$ admits a topology making it a $\sC$-algebra.
\end{lem}

\begin{proof}
It can be verified that the $*$-homomorphism $\gamma$ constructed above is an isomorphism of $*$-algebras and $\ilim_i\CT(X_i,P_i)$ is by construction a $\sC$-algebra. We may now topologise $\CT(X,P)$ via the isomorphism $\gamma$ making it a $\sC$-algebra.
\end{proof}

The association $(X,P)\mapsto \CT(X,P)$ is functorial with respect to the morphisms of pairs described above and takes values in the category of $\sC$-algebras with $*$-homomorphisms. In the spirit of \cite{RosCT} we propose the following generalization of twisted $\K$-theory to countably compactly generated and paracompact spaces.

\begin{defn}
Let $X$ be a countably compactly generated space, which is, in addition, paracompact. If $P$ is a principal $PU$-bundle on $X$, then we define the $\K$-theory of $\CT(X,P)$ to be the twisted $\K$-theory of the pair $(X,P)$.
\end{defn}

\begin{rem}
Note that the $\K$-theory we defined above applies to a $\sC$-algebra. Since the twisted $\K$-theory groups only depend on the cohomology class $[P]=\eta\in\H^3(X,\ZZ)$ determined by $P$ (up to isomorphism), one may (somewhat sloppily) refer to the twisted $\K$-theory groups of $(X,P)$ as those of of $(X,\eta)$.
\end{rem}

\begin{lem}
Twisted $\K$-theory satisfies Bott periodicity, $C^*$-stability and Milnor $\ilim^1$-sequence for an inverse limit of $\sC$-algebras.
\end{lem}

\begin{proof}
Since $\RK$-theory of $\sC$-algebras satisfies these properties, twisted $\K$-theory inherits them (see Lemma \ref{KP}).
\end{proof} The Milnor $\ilim^1$-sequence and the property \eqref{agreement}, i.e., agreement with the topological $\K$-theory of $C^*$-algebras, give a procedure to compute twisted $\K$-theory.

\begin{ex} \label{SUK}
Let $P$ be a principal $PU$-bundle on $SU(\infty)$, such that its cohomology class is $\ell\in\H^3(SU(\infty),\ZZ)\simeq\ZZ$. The inclusion $\iota_n:SU(n)\hookrightarrow SU(\infty)$, induces a homomorphism $\iota^*_n:\H^3(SU(\infty),\ZZ)\simeq\ZZ\map \ZZ\simeq\H^3(SU(n),\ZZ)$, which is identity. From the computations of Braun \cite{BraTwist} and Douglas \cite{DouTwist}, it is known that the twisted $\K$-theory groups of the pair $(SU(n),\ell)$ are finite abelian groups for all $n$ and $\ell$. More precisely, as an abelian group $$\K_\bullet(\CT(SU(n),\iota^*_n(P)))\cong\K^\bullet(SU(n),\ell)\simeq  (\ZZ/c(n,\ell))^{2^{n-1}},$$ where $\K^\bullet=\K^0\oplus\K^1$ and $c(n,\ell)=gcd\{\binom{\ell + i}{i}-1\,:\, 1\leq i\leq n-1\}$. Now, by construction, $\CT(SU(\infty),P)$ is the $\sC$-algebra $\ilim_n\CT(SU(n),\iota^*_n(P))$. One use the Milnor $\ilim^1$-sequence and the fact that an inverse system of finite abelian groups satisfies the Mittag--Leffler condition, i.e., the $\ilim^1$-term vanishes, to deduce that 
$$\K^\bullet(SU(\infty),\ell)\cong\ilim_n \K^\bullet(SU(n),\ell)\simeq\ilim_n(\ZZ/c(n,\ell))^{2^{n-1}}.$$ This is an example of a profinite group, which is easily seen to be pure torsion. Indeed, it is evident that $c(n,\ell)$ divides $c(m,\ell)$ if $n\geqslant m$. Since for a fixed $\ell$, the number $c(n,\ell)$ decreases to $0$ as $n$ increases, the profinite group $\ilim_n (\ZZ/c(n,\ell))^{2^{n-1}}$ is actually trivial.
\end{ex}

\begin{rem}
In the same vein one can construct $\CT(X,P)$ as a locally convex algebra for any (paracompact) compactly generated space $X$ and study its twisted $\K$-theory.
\end{rem}

\section{Twisted periodic cyclic homology via locally convex algebras} \label{HPT}
It is well-known that cyclic homology theory is rather poorly behaved on the category of $C^*$-algebras. In fact they vanish on the category of stable $C^*$-algebras \cite{WodFunAna}. Usually one defines cyclic homology theory with respect to a smooth $*$-subalgebra. For instance, if $X$ is a compact smooth manifold, then Connes computed the cyclic cohomology theory in terms of the Fr{\'e}chet algebra $\C^\infty(X)$ in the seminal paper \cite{ConDiff}. The periodic cyclic homology of $\C^\infty(X)$ is naturally related to de Rham cohomology of $X$ in the following manner:

\beq \label{dR}
\HP_0(\C^\infty(X))\cong\H_{dR}^\ev(X,\CC) \text{  and } \HP_1(\C^\infty(X))\cong\H_{dR}^\od(X,\CC),
\eeq where $\H^\ev = \oplus_n \H^{2n}_{dR}$ and $\H^\od =\oplus_n\H^{2n+1}_{dR}$. Let us recall some basic facts about cyclic homology theory. We follow the $\X$-complex formalism following \cite{CunQui1,CunQui2}. Some of the results stated in the generality of all locally convex algebras can be found in Chapter 4 of \cite{MeyCycHom}.

For any locally convex algebra $A$, let $A^+$ denote the unitization of $A$, which is again a locally convex algebra. Set $\Omega^0(A)=A$ and $\Omega^n(A):= A^+\grot A^{\grot n}$ for $n\geqslant 1$, which is the space of noncommutative $n$-forms. One defines $\Omega^\ev(A):=\prod_{n=0}^\infty \Omega^{2n}(A)$ and $\Omega^\od(A):=\prod_{n=0}^\infty \Omega^{2n+1}(A)$. The {\em Fedosov product} $\circ$ on $\Omega(A):=\prod_{n=0}^\infty \Omega^{n}(A)$

$$\omega\circ\eta=\omega\eta -(-1)^{kl} d\omega d\eta, \text{       for $\omega\in\Omega^k(A)$, $\eta\in\Omega^l(A)$}$$ converts $(\Omega(A),\circ)$ into an associative algebra and $\Omega^\ev(A)$ a subalgebra thereof. Let the finite product $\T_n(A):=\prod_{j=0}^n\Omega^{2j}(A)$ be the canonical quotient algebra of $\Omega^\ev(A)$ with the truncated Fedosov product, i.e., $\omega\circ\eta =0$ if $\deg(\omega)+\deg(\eta)>2n$. Then $\T_\infty(A):=\ilim_n \T_n(A)\cong\Omega^\ev(A)$ as associative algebras.

The $\X$-complex of an algebra is a very simple complex, which is useful in establishing various formal properties of periodic cyclic homology. For any algebra $A$ the following $\ZZ/2$-graded complex

\beq
A \overset{\delta}{\underset{\beta}{\rightleftarrows}} \Omega^1(A)/[A,\Omega^1(A)]
\eeq is called the $\X$-complex of $A$, denoted by $\X(A)$. Here $\delta=\pi\circ d$ with $\pi:\Omega^1(A)\map \Omega^1(A)/[A,\Omega^1(A)]$ the canonical quotient map, $d:A\map \Omega^1(A)$ the differential satisfying $d(a+b)=da + db$, $d(ab)=adb + d(a)b$, and $\beta:\Omega^1(A)/[A,\Omega^1(A)]\map A$ the map induced by the linear map $b:\Omega^1(A)\map A$ sending $adb\mapsto [a,b], db\mapsto 0, a\mapsto 0$ for all $a,b\in A$.

There is a canonical algebra homomorphism $\pi_A:\T_\infty(A)\map \T_0(A)\cong A$ given by the projection onto the first summand. Setting $\cJ_\infty(A)=\ker(\pi_A)$ one obtains an algebra extension $0\map \cJ_\infty(A)\map \T_\infty(A)\overset{\pi_A}{\map} A\map 0$, which admits a linear splitting $\sigma_A :A\map \T_\infty(A)$ given by the inclusion of $A$ onto the first summand.
An algebra $A$ is called {\em quasi-free} if this algebra extension actually splits. There are various other equivalent definitions of a quasi-free algebra. The Cuntz--Quillen formalism says that for a quasi-free algebra $A$ the periodic cyclic homology can be computed using $\X(A)$, i.e., there is a natural isomorphism $\HP_*(A)\cong \H_*(\X(A))$. Natural examples of quasi-free algebras are not abundant; however, $\CC$ and $M_n(\CC)$ are quasi-free. Moreover, there is a functorial way to manufacture a quasi-free algebra from any algebra. For any locally convex algebra $A$, the algebra $\T_\infty(A)$ constructed above is always quasi-free, whence $\HP_*(A)\cong\H_*(\X(\T_\infty(A)))$. In view of the above natural isomorphism, one can take the right hand side as the definition of periodic cyclic homology of a locally convex algebra.

\begin{defn}
Given any two locally convex algebras $A,B$ one defines the bivariant periodic cyclic homology groups as $$\HP_*(A,B)= \H_*(\Hom(\X(\T_\infty(A)),\X(\T_\infty(B)))).$$

\noindent
In particular, the periodic cyclic homology (resp. cohomology) groups of $A$ are defined as
$$ \text{$\HP_*(A)=\HP_*(\CC,A)$ (resp. $\HP^*(A)=\HP_*(A,\CC)$).}  $$
\end{defn}
\noindent
Here $\Hom(\X(-),\X(?))$ is the mapping Hom-complex between $\ZZ/2$-graded complexes (see subsection \ref{locHom} below for some more details).

\subsection{Periodic cyclic homology of countable inductive limit of smooth compact manifolds}
Let $X=\dlim_n X_n$ be a countably compactly generated space. Then $\C(X)=\ilim_n \C(X_n)$ is a $\sC$-algebra. Now suppose, in addition, that each $X_n$ is a smooth compact manifold and that the connecting maps $f_n:X_n\map X_{n+1}$ are also smooth, then $\{\C^\infty(X_n),f^*_n\}$ forms a countable inverse system of locally multipicatively convex (Fr{\'e}chet) algebras and continuous homomorphisms. Let us set $\C_f^\infty(X)=\ilim_n \C^\infty(X_n)$ as the algebra of {\em formal smooth functions}, where the inverse limit is constructed in the category of locally convex algebras. Being an inverse limit of complete locally multiplicatively convex algebras, $\C_f^\infty(X)$ is also a complete locally multiplicatively convex algebra (see Page 84 of \cite{Mallios}). We may now define the (formal) periodic cyclic homology of $X$ as $\HP_*(\C_f^\infty(X))$, where $\C_f^\infty(X)$ is viewed as a locally multipicatively convex algebra. Let $\{A_n,f_n\}$ be a countable inverse system of topological algebras and 
continuous homomorphisms. The inverse system is called {\em reduced} if the canonical maps $\ilim_n A_n\map A_m$ have dense range for all $m\in\NN$. The periodic cyclic homology of the inverse limit $A=\ilim_n A_n$ of a reduced countable inverse system of Fr{\'e}chet algebras can be computed from the following short exact sequence (see Theorem 5.4 of \cite{Lyk})
\beq \label{lim1}
0\map{\ilim}_n^1\HP_{*+1}(A_n)\map\HP_*(A)\map\ilim_n\HP_*(A_n)\map 0.
\eeq

\begin{lem}
Let $X_0\overset{f_0}{\hookrightarrow} X_1\overset{f_1}{\hookrightarrow} \cdots$ be a countable sequence of (smooth) inclusions of compact smooth manifolds. Then $\HP_*(\C^\infty_f(X))\cong\ilim_n\HP_*(\C^\infty(X_n))$.
\end{lem}

\begin{proof}
It follows from Lemma 1.4. of \cite{PhiFreK} that $\{\C^\infty(X_n),f^*_n\}$ is a reduced countable inverse system of Fr{\'e}chet algebras. Applying equation \eqref{lim1} to the reduced countable inverse system $\{\C^\infty(X_n),f^*_n\}$, we get

$$0\map{\ilim}_n^1\HP_{*+1}(\C^\infty(X_n))\map\HP_*(\C^\infty_f(X))\map\ilim_n\HP_*(\C^\infty(X_n))\map 0.$$ From equation \eqref{dR} we conclude that $\{\HP_*(\C^\infty(X_n)),f^*_n\}$ is a countable inverse system of finite dimensional real vector spaces, whence the Mittag--Leffler condition is satisfied. Consequently the $\ilim^1$-term vanishes and the result follows.
\end{proof}

Note that in the above discussion we did not claim that the inductive limit space $X$ admits a reasonable smooth manifold structure. The two specific examples $U(\infty), SU(\infty)$ are, of course, infinite dimensional Lie groups (regular in the sense of Milnor) modelled on {\em convenient topological vector spaces} \cite{KriMic}. In these examples each connecting map $f_n$ is actually a smooth inclusion of Lie groups. In fact, it is known that if each $X_n$ is a finite dimensional smooth manifold and each $f_n$ is a smooth immersion, then the topological direct limit $X=\dlim_n X_n$ can be endowed with a smooth structure turning $X$ into a (possibly infinite dimensional) smooth manifold, such that $X$ is the direct limit in the category of smooth manifolds (modelled on topological vector spaces) and smooth maps (see Theorem 3.1 of \cite{GloInfDimLie}). Therefore, one can define the algebra $\C^\infty(X)$ of genuine smooth functions on $X$.

\begin{rem}
In the absence of a smooth structure on $X$, the algebra of formal smooth functions $\C^\infty_f(X)$ (resp. its periodic cyclic homology $\HP_*(\C^\infty_f(X))$) is potentially a good replacement for the algebra of genuine smooth functions and the $\ZZ/2$-periodic version of the de Rham cohomology groups.
\end{rem}

\begin{prop}
Let $X_0\overset{f_0}{\hookrightarrow} X_1\overset{f_1}{\hookrightarrow} \cdots$ be a countable sequence of (smooth) inclusions of compact smooth manifolds, such that the inductive limit $X=\dlim_n X_n$ exists as a smooth manifold. Then there is an algebra isomorphism $\C^\infty(X)\cong\C_f^\infty(X)$.
\end{prop}

\begin{proof}
The canonical inclusions $i_n:X_n\map X$ are smooth and they produce algebra homomorphisms $i_n^*:\C^\infty(X)\map\C^\infty(X_n)$, which actually assemble to produce a map to the inverse system $\{\C^\infty(X_n),f_n^*\}$. Consequently there is an induced map to the inverse limit, i.e., $i=\ilim_n i_n: \C^\infty(X)\map\ilim_n\C^\infty(X_n)$. It follows from the fact that $X$ is the inductive limit of $\{X_n,f_n\}$ in the category of smooth manifolds with smooth maps, that $\ilim_n i_n$ is an isomorphism of algebras.
\end{proof}

 Now we endow $\C^\infty(X)$ with the locally multiplicatively convex topology via the above isomorphism, so that $\HP_*(\C^\infty(X))\cong\ilim_n\HP_*(\C^\infty(X_n))$.

 \begin{ex}
 Let $\H^\bullet_{dR}$ (resp. $\HP_\bullet$) denote $\H^\ev_{dR}\oplus \H^\od_{dR}$ (resp. $\HP_0\oplus \HP_1$). It is well known that $\H_{dR}^\bullet(SU(n),\CC)\simeq \Lambda_\CC(x_3,\cdots, x_{2n-1})$ with $x_{2i-1}\in\H^{2i-1}_{dR}(SU(n),\CC)$ (here $\Lambda_\CC$ denotes the complex exterior algebra). The canonical inclusion $SU(n-1)\hookrightarrow SU(n)$ induces a homomorphism $\Lambda_\CC(x_3,\cdots,x_{2n-1})\map\Lambda_\CC(x_3,\cdots,x_{2n-3})$, which simply kills the generator $x_{2n-1}$. Therefore, $$\HP_\bullet(\C^\infty(SU(\infty)))\cong\ilim_n\HP_\bullet(\C^\infty(SU(n)))\cong\ilim_n\H^\bullet_{dR}(SU(n),\CC)\simeq\ilim_n \Lambda_\CC(x_3,\cdots, x_{2n-1}),$$ which agrees with the cohomology of $SU(\infty)$ with complex coefficients. The other examples, viz., $U(\infty)$ and $\CC P^\infty$ can be computed similarly.
 \end{ex}

 \subsection{Twisted periodic cyclic homology} \label{HPTwist}
 Now we study the twisted version of periodic cyclic homology. The idea is to produce a locally convex algebra from the given twisting data and define the periodic cyclic homology of that algebra as the twisted periodic cyclic homology. Here we only talk about the formal analogue of the previous subsection and make no attempt to produce a {\em genuine noncommutative twisted smooth space}.

Let $H$ be a separable infinite dimensional Hilbert space and let $\{v_n\}_{n\in\NN}$ be an orthonormal basis of $H$. Let $B(H)$ denote the $C^*$-algebra of bounded operators on $H$. One defines a semi-finite trace on the positive elements in $B(H)$ as $\tr(x)=\sum_{n=0}^\infty(xv_n,v_n)$. This trace is independent of the choice of the basis $\{v_n\}$ and satisfies $\tr(x^*x)=\tr(xx^*)$. For $p\geqslant 1$, one defines the $p$-Schatten ideal $\cL^p\subset B(H)$ as $\cL^p=\{x\in B(H)\,|\, \tr(\sqrt{x^*x})^p <\infty\}$. It is a Banach $*$-ideal with respect to the norm $\|x\|_p=(\tr(\sqrt{x^*x})^p)^{1/p}$. For all $p$, $\cL^p\subset \cpt$. By definition $\cL^1$ is the Banach $*$-ideal of {\em trace class operators}, since for this ideal the trace is finite for all elements.

As we discussed before, $C^*$-algebras and $\sC$-algebras are not the appropriate geometric objects for the study of cyclic homology theories. Therefore, we need to modify the construction of the continuous trace $\sC$-algebra $\CT(X,P)$ that was used to define the twisted $\K$-theory of the pair $(X,P)$. The authors in \cite{MatSte} provided a candidate roughly by replacing the algebra of compact operators $\cpt$ by the Banach $*$-ideal of trace class operators $\cL^1$ as described above. Given a principal $PU$-bundle $P$ on a compact smooth manifold $X$, one can form a Banach algebra bundle $\cL^1(P)=P\times_{PU}\cL^1$ with the algebra of trace class operators $\cL^1$ as the fibre associated to $P$ via the adjoint action of $PU$ on $\cL^1$. The algebra of (smooth) sections $\C^\infty(\cL^1(P))=\C^\infty(X,\cL^1(P))$ can be endowed with a Fr{\'e}chet algebra structure. For the details we refer the readers to page 308 of ibid..

Let $X_0\hookrightarrow X_1\hookrightarrow \cdots$ be a countable directed system of compact smooth manifolds and smooth immersions, such that the inductive limit is also a smooth and paracompact manifold, e.g., $SU(\infty)$. Let $P$ be a principal $PU$-bundle on $X$, whose isomorphism class determines an element in $\H^3(X,\ZZ)$. As before, by passing to a cofinal subsystem, we write $X=\cup_n X_n$  with each $X_{n-1}\subset X_n$ being a smooth inclusion of smooth compact manifolds. We denote the restricted bundle $P|_{X_n}$ on each $X_n$ by $P_n$. One constructs the Fr{\'e}chet algebras $\C^\infty(\cL^1(P_n))=\C^\infty(X_n,\cL^1(P_n))$ as described above and there are canonical restriction homomorphisms $\C^\infty(\cL^1(P_n))\map\C^\infty(\cL^1(P_{n-1}))$. Once again it follows from Lemma 1.4. of \cite{PhiFreK} that $\{\C^\infty(\cL^1(P_n))\}$ is a reduced inverse system of Fr{\'e}chet algebras. Let us define the {\em twisted smooth algebra} of $(X,P)$ as $$\C^\infty(X,\cL^1(P)) =\C^\infty(\cL^1(P))=\ilim_
n \C^\infty(\cL^1(P_n)).$$ Being an inverse limit of locally convex (Fr{\'e}chet) algebras, it is itself a locally convex algebra and can be regarded as a genuinely {\em noncommutative smooth space}. Now we may define the twisted periodic cyclic homology of $(X,P)$ as $$\HP_*(\C^\infty(\cL^1(P)))=\H_*(\X(\cT_\infty(\C^\infty(\cL^1(P))))).$$ We shall refer to the $\X$-complex $\C^\infty(\cL^1(P))$ as the {\em twisted $\X$-complex}.

\begin{ex} \label{SUHP}
It is well-known that if $A$ is a $C^*$-algebra and $\A\subset A$ is a dense subalgebra, which is closed under the holomorphic functional calculus, then the inclusion $\A\hookrightarrow A$ induces an isomorphism $\K_*(\A)\cong\K_*(A)$ (see \cite{ConDiff}, also Corollary 7.9 of \cite{PhiFreK}). Let $X=\dlim_n X_n$ and $P$ be as above. It is shown in section 4.2 of \cite{MatSte} that $\C^\infty(\cL^1(P_n))$ is a dense subalgebra of the continuous trace $C^*$-algebra $\CT(X_n,P_n)$, which is closed under the holomorphic functional calculus, whence their $\K$-theory groups agree.

Let $X=SU(\infty)$ and let $P$ be a principal $PU$-bundle on it. Set $\cB_n=\C^\infty(\cL^1(P_n))$. It is known that the twisted Chern--Connes character map $\ch_*(B_n):\K_*(\cB_n)\map\HP_*(\cB_n)$ becomes an isomorphism after tensoring with the complex numbers (see Proposition 6.1 of \cite{MatSte}). Since the twisted $\K$-theory groups are all torsion (see Example \ref{SUK} above), we conclude that the twisted periodic cyclic homology groups vanish, i.e., $\HP_*(\cB_n)=\{0\}$ for all $n$. Using the Milnor $\ilim^1$-exact sequence in periodic cyclic homology (see Equation \eqref{lim1}), we immediately deduce that the twisted periodic cyclic homology groups of $(SU(\infty),P)$ vanish as well.
\end{ex}

\section{Twisted Chern--Connes character} \label{HPCC}
Restricted to the category of $m$-algebras, the bivariant $\K$-theory for locally convex algebras described above has a universal characterization. Using this universal characterization the author constructed (see Theorem 21.2 of \cite{CunBivCh}) a multiplicative bivariant Chern--Connes character $$\ch^\biv_*:\kk_*\functor\HP_*.$$ Setting the first variable to $\CC$, we get the univariant Chern--Connes character from $\K$-theory to periodic cyclic homology $\ch_*(A):\K_*(A)\map\HP_*(A)$ for any $m$-algebra $A$. We exhibit an example below, where this $\HP_*$-valued Chern--Connes character is not an isomorphism after tensoring with the complex numbers. The problem is that the groups involved are themselves not isomorphic - the twisted $\K$-theory is nontrivial, whereas the periodic cyclic homology is trivial.
\begin{ex} \label{NTriv}
We have seen that the twisted $\K$-theory of $(SU(\infty),P)$, where $P$ is a principal $PU$-bundle on $SU(\infty)$, is the trivial group (see Example \ref{SUK}). Therefore, it is also trivial after tensoring with $\CC$ and so is periodic cyclic homology.

Consider the pair $(S^3,P^m)$, where $S^3$ denotes the $3$-sphere and $P^m$ is a principal $PU$-bundle on $S^3$, whose cohomology class is $m\in\H^3(S^3,\ZZ)\simeq\ZZ$. Let $(X,P)$ be the pair, where $X=\coprod_n X_n$ is a countable disjoint union space with $X_n=S^3$ for all $n$ and $P$ restricts to the principal $PU$-bundle $P^n$ on each $X_n$. The associated $\sC$-algebra is $\prod_n\CT(X_n,P_n)$. It is known (see \cite{RosHomInv,RosCT}) that the $\K$-theory of the continuous trace $C^*$-algebra $\CT(X_n,P^n)$ is

\beqn
\K_*(\CT(X_n,P^n)) =
\begin{cases} 0 \text{ if $*=0$,}\\
                          \ZZ/n \text{ if $*=1$}
\end{cases}
\eeqn After tensoring with $\CC$ it becomes isomorphic to the twisted periodic cyclic homology of $(X_n,P^n)$ via the twisted Chern--Connes character map. Therefore, twisted $\HP$-theory vanishes for each $(X_n,P^n)$ and since it commutes with countable products, it vanishes for $(X,P)$. However, from property \eqref{product} of (representable) $\K$-theory we conclude that the twisted $\K$-theory of the pair $(X,P)$ is

 \beqn
\K_*(\CT(X,P)) =
\begin{cases} 0 \text{ if $*=0$,}\\
                          \prod_n \ZZ/n \text{ if $*=1$,}
\end{cases}
\eeqn where clearly the twisted $\K_1$-group does not vanish after tensoring with $\CC$. Indeed, $\prod_n\ZZ/n\supset \hat{\ZZ}$, the ring of profinite integers, which gives rise to $\mathbb{A}_f$ after tensoring with $\QQ$. Here $\mathbb{A}_f$ denotes the ring of finite adeles, which is a well-known object in number theory. Therefore, the $\HP_*$-valued Chern--Connes character cannot be an isomorphism after tensoring with the complex numbers.
\end{ex}

\subsection{$\HL_*$-valued Chern--Connes character in the compact case} \label{Pusch}
The main advantage of local cyclic homology $\HL$-theory \cite{Puschnigg,PuschniggHL} is that it gives a satisfactory answer for $C^*$-algebras in the following sense: If $M$ is a smooth compact manifold then the natural inclusion $\C^\infty(M)\hookrightarrow\C(M)$ induces an $\HL$-isomorphism. In order to define the $\HP$-valued Chern--Connes character one needed to extract a suitable dense smooth subalgebra of a $C^*$-algebra or a $\sC$-algebra with the same $\K$-theory and then define the map. If one is interested in the $\HL$-valued Chern--Connes character, then one can directly work with $C^*$-algebras.

We can exploit the $\X$-complex formalism for local cyclic homology as well. For a Banach algebra $A$ one constructs an {\em analytic tensor algebra} $\T_\an(A)$ similar to $\T_\infty(A)$ as in Section \ref{HPT}, but completed with respect to a {\em bornology}. Then one defines the {\em local homology} of the $\ZZ/2$-periodic $\X$-complex $\X(\T_\an(A))$ to be the local cyclic homology groups of $A$ (see also Definition \ref{HLdefn} below). Since the construction of the more general bivariant local cyclic homology groups are discussed in Section \ref{HL} below, we do not explain the details here. Let $X$ be a compact space and let $P$ be a principal $PU$-bundle on $X$. One constructs the continuous trace $C^*$-algebra $\CT(X,P)$ as before and defines the {\em twisted local cyclic homology} groups of $(X,P)$ as $\HL_*(\CT(X,P))$. Now we show that if $X$ is, in addition, a smooth manifold, then the above definition of the twisted local cyclic homology groups will agree with those of the smooth modification in 
terms of $\C^\infty(X,\cL^1(P))$.

 \begin{thm} \label{HLcpt}
 Let $X$ be a compact smooth manifold (possibly with boundary) and $P$ be a principal $PU$-bundle on $X$. Then the canonical homomorphism $\C^\infty(\cL^1(P))\map\CT(X,P)$ induces isomorphisms $\HL_*(\C^\infty(\cL^1(P)))\cong\HL_*(\CT(X,P))$ and $\HL^*(\C^\infty(\cL^1(P)))\cong\HL^*(\CT(X,P)).$
 \end{thm}

 \begin{proof}
Let us first prove the isomorphism in local cyclic homology. If $X$ is contractible then $P$ is trivializable, whence $\C^\infty(\cL^1(P))\cong\C^\infty(X)\grot \cL^1$ and $\CT(X,P)\cong\C(X)\prot\cpt$. Consider the commutative diagram

 \beq
 \xymatrix{
 \C^\infty(X)\ar[r]\ar[d] & \C^\infty(X)\grot\cL^1\ar[d] \\
 \C(X)\ar[r] &\C(X)\prot \cpt ,
 }
 \eeq where the left vertical arrow induces an $\HL$-isomorphism. Since local cyclic homology is $\cL^1$-stable (see Theorem 5.65 of \cite{MeyCycHom}) and $C^*$-stable when restricted to the category of $C^*$-algebras (see Theorem 6.25 of ibid.), it follows that the top and the bottom horizontal arrows are $\HL$-isomorphisms as well. As a consequence the right vertical arrow $\C^\infty(X)\grot\cL^1\map\C(X)\prot\cpt$ is an $\HL$-isomorphism.

We claim that $\HL_*$ is a homology theory on the category of $C^*$-algebras. Indeed, thanks to excision in $\HL$-theory, the functors $\HL_*$ are split exact on the category of $C^*$-algebras. Any $C^*$-stable and split exact functor on the category of $C^*$-algebras is automatically homotopy invariant \cite{Hig2}, whence $\HL_*$ is homotopy invariant. Therefore, using Theorem 21.2.2 of \cite{Blackadar} one concludes that Mayer--Vietoris property holds for $\HL$-theory. Since $X$ is a compact smooth manifold one can choose a finite cover consisting of geodesically convex open balls $\{B_\alpha\}$ with $\overline{B_\alpha}\subset X$ compact and geodesically convex (hence contractible). The proof goes by induction on the number of subsets in the cover. The base case of induction is covered by the previous paragraph. Using the fact that finite non-empty intersection of geodesically convex subsets is again geodesically convex, one may invoke the Mayer--Vietoris sequence and the $5$-Lemma to conclude the general 
result.

\noindent
The proof for the isomorphism in local cyclic cohomology is similar and hence omitted.
 \end{proof}

 Puschnigg constructed a bivariant Chern--Connes character $\ch^P:\KK_*(A,B)\map\HL_*(A,B)$, where $A$ and $B$ are separable $C^*$-algebras. It is known that $\HL$-theory has a composition product, so that one can define an additive $\HL$-category, denoted by $\mathtt{HL}_{\mathtt{C^*}}$, with the bivariant $\HL_0$-groups as morphisms. The existence of the bivariant Chern--Connes character in degree $0$ follows from the characterization of Kasparov's bivariant $\KK$-category, denoted by $\KKcat$, as the universal $C^*$-stable and split exact functor \cite{Hig1}. Let $\Csep$ denote the category of separable $C^*$-algebras with canonical functors $\KK:\Csep\functor\KKcat$ and $\HL:\Csep\functor\mathtt{HL}_{\mathtt{C^*}}$.

 \begin{thm}[Puschnigg, Theorem 6.3 of \cite{Puschnigg}] \label{PCC} The bivariant Chern--Connes character is uniquely characterized by the following two properties:

 \begin{itemize}
 \item If $f:A\map B$ is a $*$-homomorphism of $C^*$-algebras, then $\ch^P(\KK(f))=\HL(f)$.

 \item Let $[\epsilon]\in\KK_1(A,B)$ be represented by an extension diagram $0\map B\prot\cpt \map C\map A\map 0$, admitting a completely positive contractive linear section $A\map C$. Since $\HL$-theory satisfies excision with respect to such extensions, one obtains a class $[\delta]\in\HL_1(A,B\prot\cpt)\cong\HL_1(A,B)$. Then $\ch^P([\epsilon])=[\delta]$.

 \end{itemize}
 \end{thm}

 \noindent
 Setting $A=\CC$ in the bivariant Chern--Connes character one obtains the univariant Chern--Connes character $\ch_*(A):\K_*(A)\map\HL_*(A)$.

 \begin{prop}
 Let $X$ and $P$ be as above. Then the twisted Chern--Connes character map $\ch_*(\CT(X,P)):\K_*(\CT(X,P))\map\HL_*(\CT(X,P))$ becomes an isomorphism after tensoring with $\CC$.
  \end{prop}

\begin{proof}
The assertion follows from Theorem 7.7 of \cite{MeyCycHom}, since the $C^*$-algebra $\CT(X,P)$ belongs to the bootstrap category, which can be defined as the category of separable $C^*$-algebras satisfying the universal coefficient theorem (UCT). Indeed, $\CT(X,P)$ is a type $I$ $C^*$-algebra (see, for instance, Theorem 6.1.11. of \cite{Pedersen}) and such algebras satisfy UCT \cite{RosSch}.
  \end{proof}

\begin{cor}
The twisted cohomology (with complex coefficients) of the pair $(X,P)$ is isomorphic to $\HL_*(\CT(X,P))$.
\end{cor}

\begin{proof}
There is a chain of isomorphisms $$\HL_*(\CT(X,P))\cong\K_*(\CT(X,P))\otimes\CC\cong\K_*(\C^\infty(\cL^1(P)))\otimes\CC\cong\HP_*(\C^\infty(\cL^1(P))),$$ where the last identification was shown in Proposition 6.1. of \cite{MatSte}. Moreover, $\HP_*(\C^\infty(\cL^1(P)))$ was identified with the twisted cohomology of $(X,P)$ in Proposition 6.3. of ibid..
\end{proof}

 \section{Twisted $\K$-homology via separable $\sC$-algebras} \label{KHom}
 Atiyah--Hirzebruch complex $\K$-theory has a dual theory, which is intimately connected to index theory. It is called $\K$-homology theory and in noncommutative geometry its analytic version is seen as a special case of Kasparov's bivariant $\K$-theory \cite{KasKK1,KasKK2}, viz., for any separable $C^*$-algebra $A$ the analytic $\K$-homology is defined to be $\K^*(A)=\KK_*(A,\CC)$. It is a {\em $\sigma$-additive cohomology theory} on the category of separable and nuclear $C^*$-algebras (see \cite{KasNovikov,RosSch}). The class of separable $C^*$-algebras is not too restrictive; separability imposes a metrizability condition on the spectrum of a commutative $C^*$-algebra. 

We need an extension of Kasparov's bivariant $\K$-theory to the category of sepable $\sC$-algebras. The bivariant $\K$-theory for locally convex algebras, described in Section \ref{CKTh}, does not agree with Kasparov's theory, when restricted to the category of separable $C^*$-algebras. The problem arises in the $\K$-homology part. In keeping with the exposition so far, we present a modified bivariant $\K$-theory for separable $\sC$-algebras, suggested by Cuntz \cite{CunGenBivK}, which is a generalization of Kasparov's bivariant $\K$-theory.

Let $\cU$ be a full subcategory of locally convex $*$-algebras, which satisfies the following set of axioms:

\begin{enumerate}

\item $\kh$ For each $A\in\cU$ there is a functorial cylinder algebra $A[0,1]\in\cU$ with two natural continuous evaluation homomorphisms $A[0,1]\map A$ ($\ev_0$ and $\ev_1$). One can formulate the notion of a homotopy between two morphisms in $\cU$ as follows: Two morphisms $f,g:A\map B$ in $\cU$ are said to be {\em homotopic} if there exists a morphism $h: A\map B[0,1]$ in $\cU$, such that $f=\ev_0\circ h$ and $g=\ev_1\circ h$.

\item $\ks$ For each $A\in\cU$ there is a functorial stabilized algebra $\cpt(A)$ containing $M_\infty(A)$ and a corner embedding $\iota: A\map \cpt(A)$, such that the canonical map $\cpt(\iota):\cpt(A)\map\cpt(\cpt(A))$ is homotopic to a specific isomorphism $\cpt(A)\cong\cpt(\cpt(A))$. Note that a homotopy between two morphisms is defined using the previous axiom $\kh$.

\item $\knew$ There is fixed choice $\cL$ of a subclass of all continuous linear morphisms between the objects of $\cU$, which contains all the morphisms of $\cU$,  such that for any $A\in\cU$ there is a fixed map $s: A\map TA$ in $\cL$ with the property that given any other map $\alpha: A\map B$ belonging to $\cL$, there exists a unique morphism $\beta: TA\map B$ in $\cU$ satisfying $\alpha=\beta\circ s$.

\item $\ke$ There is a distinguished class $\fC$ of extensions in $\cU$, which are split by a map in $\cL$. We only require the existence of such a splitting; the choice of a splitting is not a part of the data. Now for each $A\in\cU$ the following extensions must be in $\fC$.

\begin{itemize}
\item a functorial cone-suspension extension: $0\map A(0,1)\map A(0,1]\map A\map 0$. Here the suspension $A(0,1)$ and the cone $A(0,1]$ have the obvious definitions in terms of $\ev_0$, $\ev_1$ and the cylinder $A[0,1]$.

\item a functorial (reduced) Toeplitz extension: $0\map\cpt(A)\map\fT(A)\map A(0,1)\map 0$.

\item a (homotopy) universal extension: The map $A\overset{\id}{\map} A$ produces a surjective algebra homomorphism $TA\map A$ in $\cU$ by $\knew$. We require the extension $$0\map JA\map TA\map A\map 0$$ to be in $\fC$, where $JA:=\ker(TA\map A)$, with its canonical splitting $s: A\map TA$ in $\cL$. Given any extension $0\map I\map E\map A\map 0$ in $\fC$, there is a (not necessarily unique) morphism of extensions

\beqn
\xymatrix{
0\ar[r] & JA\ar[r]\ar[d]^{\epsilon_A} & TA\ar[r]\ar[d] & A\ar[r]\ar[d]^\id & 0\\
0\ar[r] & I \ar[r] & E \ar[r] & A\ar[r] & 0,}
\eeqn which is obtained by choosing a splitting $A\map E$ in $\cL$. The morphism $\epsilon_A$ is called a {\em classifying map} of the extension $0\map I\map E\map A\map 0$ and it is required to be unique up to a homotopy.
\end{itemize}

\item $\kt$ There is an associative tensor product $\otimes$ on $\cU$, such that tensoring with any object in $\cU$ preserves the extensions in $\fC$.
\end{enumerate}

\begin{rem}
Our set of axioms is a bit more restrictive than that of Cuntz in \cite{CunGenBivK}. For instance, the axiom $\knew$ does not appear in ibid.. It produces a stronger version of the (homotopy) universal extension $0\map JA\map TA\map A\map 0$, which will be functorial in $A$. The axiom $\ks$, which is needed to make sense of stability, has also been strengthened. However, in the example, that we are going to consider, this stronger set of axioms will be satisfied. We have left out a predictable `subdivision of the unit interval condition' that is needed to ensure that homotopy is an equivalence relation.
\end{rem}

\noindent
The choice of the class of extensions $\fC$ in the axiom $\ke$ determines the behaviour of the bivariant theory, particularly its excisive properties. Splicing the cone-suspension extension with the (reduced) Toeplitz extension one gets a $2$-step extension diagram:

$$0\map\cpt(A)\map \fT(A)\map A(0,1]\map A\map 0.$$ Using the property of the universal extension (twice) one gets a {\em classifying map} $\epsilon_A: J^2(A)\map\cpt(A)$ of the $2$-step extension. Let $[A,B]$ denote the homotopy classes of homomorphisms between $A$ and $B$. Then $\epsilon_A$ defines a map $S:[J^k(A),\cpt(B)]\map [J^{k+2}(A),\cpt(B)]$, which takes any homotopy class $[\alpha]\in [J^k(A),\cpt(B)]$ to the homotopy class of the map $$J^{k+2}(A)=J^2(J^k(A))\overset{\epsilon_{J^k(A)}}{\functor} \cpt(J^k(A))\overset{\cpt (\alpha)}{\functor}\cpt(\cpt(B))\cong\cpt(B).$$ For $i=0,1$, one now defines the bivariant $\K$-theory groups as

\beq \label{kk}
\kk_i(A,B)=\dlim_n[J^{2n+i}(A),\cpt(B)],
\eeq where $[-,-]$ denotes the homotopy classes of morphisms and the direct limit is taken over the maps $S$ described above. There is a natural abelian group structure on $\kk_i(A,B)$, which is contravariantly functorial in $A$ and covariantly functorial in $B$. Furthermore, there is an associative bilinear composition product $\kk_i(A,B)\times\kk_j(B,C)\map\kk_{i+j}(A,C)$, which enables us to construct an additive category $\kk(\cU)$, whose objects are those of $\cU$ and morphisms spaces are the $\kk_0$-groups. For all the technical details concerning these assertions we refer the readers to the original articles of Cuntz \cite{CunGenBivK,CunBivCh}. We only recall the universal characterization of this bivariant $\K$-theory that will be needed in the sequel.

\begin{thm}[Cuntz, Proposition 1.2 of \cite{CunGenBivK}] \label{UP}
There is canonical functor $\kk_0:\cU\functor\kk(\cU)$, which is identity on objects, satisfying the following properties:

\begin{itemize}
\item $(E1)$ the evaluation homomorphisms $\ev_i:A[0,1]\map A$ are mapped to isomorphisms, i.e., $\kk_0(\ev_i)$ is an isomorphism for $i=0$ and $1$,

\item $(E2)$ the corner embedding $\iota: A\map\cpt(A)$ is mapped to an isomorphism, i.e., $\kk_0(\iota)$ is an isomorphism,

\item $(E3)$ for any extension $0\map I\map A\map B\map 0$ in $\fC$, and any $D\in\cU$, there is a six-term exact sequence

\beqn
\xymatrix{
\kk_0(D,I)\ar[r] & \kk_0(D,A)\ar[r] & \kk_0(D,B)\ar[d]\\
\kk_0(D,B(0,1)) \ar[u] & \kk_0(D,A(0,1))\ar[l] & \kk_0(D,I(0,1)) \ar[l]
}\eeqn and, similarly, a six-term exact sequence for the functor $\kk_0(-,D)$ with the arrows reversed. Furthermore, it is the universal functor in the following sense: Let $F:\cU\map\cC$ be any covariant additive category valued functor, so that $F(f\circ g)=F(f)\cdot F(g)$, and such that for all $D\in\cU$ the functors $\Hom_\cC(F(-),F(D))$ and $\Hom_\cC(F(D),F(-))$ satisfy the properties $(E1)$, $(E2)$ and $(E3)$. Then there is a unique covariant functor $F':\kk(\cU)\functor\cC$, such that $F= F'\circ\kk_0$.
\end{itemize}
\end{thm}

\begin{rem}
One actually needs to also assume that $F(\Hom_\cU(\fT(A),\fT(A)))=\{0\}$ for all $A\in\cU$ in the above assertion; this will be automatically satisfied in our applications below.
\end{rem}

\begin{rem} \label{E3form}
If $\cU$ is the category of (separable) $C^*$-algebras (resp. $\sC$-algebras) and $\fC$ consists of all (separable) $C^*$-algebra (resp. $\sC$-algebra) extensions, admitting a completely positive contractive linear section, then the property $(E3)$ above can be relaxed to

\begin{itemize}
\item $(E3')$ for any extension $0\map I\map A\map B\map 0$ in $\fC$, and any $D\in\cU$, the sequences $F(D,I)\map F(D,A)\map F(D,B)$ and $F(B,D)\map F(A,D)\map F(I,D)$ are exact.
\end{itemize} This is because standard arguments from the theory of operator algebras, which also works for $\sC$-algebras (see, for instance, Corollary 2.5. of \cite{PhiRepK}), enable us to extend the short exact sequence $F(D,I)\map F(D,A)\map F(D,B)$ to a long exact sequence (using $(E1)$)

\beqn
\cdots\map F(D,B(0,1)^2)&\map& F(D,I(0,1))\map F(D,A(0,1))\map\\
&\map& F(D,B(0,1))\map F(D,I)\map F(D,A)\map F(D,B)
\eeqn and Cuntz's proof of Bott periodicity in the presence of axioms $(E1)$, $(E2)$ and $(E3')$ applies, whence one gets the desired six-term exact sequence of property $(E3)$. Exactly similar arguments hold for the contravariant functor $F(-,D)$. This equivalent formulation of property $(E3)$ will be useful for us later.
\end{rem}

A $\sC$-algebra $A\cong\ilim_{n\in\NN} A_n$ is called {\em separable} if each $A_n$ is a separable $C^*$-algebra. For our purposes we are going to set $\cU$ to be the category of separable $\sC$-algebras $\sCalg$. Let $A\cong\ilim_n A_n$ and $B\cong\ilim_m B_m$ be two $\sC$-algebras. Then the maximal tensor product is defined as $A\prot_{\textup{max}} B=\ilim_{n} A_n\prot_{\textup{max}} B_n$. Henceforth, we only consider the maximal tensor product and, for brevity, write simply $\prot$ instead of $\prot_{\textup{max}}$. In particular, if $A$ is a $\sC$-algebra, then the cylinder, the cone, the suspension, and the stabilization have predictable choices, viz., $A[0,1]=\C([0,1])\prot A$, $A(0,1] =\C_0((0,1])\prot A$, $A(0,1)=\C_0((0,1))\prot A$, and $\cpt(A)=\cpt\prot A$ respectively. If $A$ is a $\sC$-algebra and $I$ is a closed two-sided $*$-ideal then $A/I$ is automatically a $\sC$-algebra and any $*$-homomorphism $A\map C$ of $\sC$-algebras that vanishes on $I$ factors through $A/I$, i.e., $A/I$ is a 
categorical quotient. If $A$ were a pro $C^*$-algebra, then $A/I$ might fail to exist as a pro $C^*$-algebra. A sequence of $\sC$-algebras and $*$-homomorphisms $0\map I\map A\map B\map 0$ is called {\em exact} if it is algebraically exact, the map $I\map A$ has a closed range and is a homeomorphism onto its image, and the induced map $A/I\map B$ is also a homeomorphism. It turns out that the topological conditions are redundant and such a sequence is exact if and only if it is algebraically exact (see Corollary 5.5 of \cite{NCP1}). It is known that $$0\map \ilim_n I_n\cong I\map \ilim_n E_n\cong E\map \ilim_n A_n\cong A\map 0$$ is a $\sC$-algebra extension if and only if $\{0\map I_n\map E_n\map A_n\map 0\}$ is an inverse system of $C^*$-algebra extensions (see Proposition 5.3 (2) of \cite{NCP1}).

Now let $\cU$ be the category of separable $\sC$-algebras. One observes at once that there is an immediate candidate for the object $A[0,1]$ (resp. $\cpt(A)$), viz., $\C([0,1])\prot A$ (resp. $\cpt\prot A$) with the obvious evaluation homomorphisms (resp. corner embedding). These constructions are clearly functorial and hence $\cU$ satisfies axioms $\kh$ and $\ks$. We choose the distinguished class of extensions $\fC$ to be all extensions $0\map I\map E\map A\map 0$ in $\cU$, i.e., exact sequences of $\sC$-algebras, admitting a completely positive contractive linear section, i.e., $\cL$ is the subclass of all completely positive contractive linear maps. For the extension preserving tensor product, that is required in the axiom $\kt$, we choose the maximal tensor product $\prot$. The cone-suspension extension and the (reduced) Toeplitz extension clearly belong to $\fC$. For any $\sC$-algebra $A$, the (reduced) Toeplitz extension is obtained by applying the functor $-\prot A$ to the extension $0\map\cpt\map\fT_
0\map \C_0(S^1\setminus\{1\})\cong\C_0((0,1))\map 0$, i.e., $\fT(A)=\fT_0\prot A$. Let us now construct a universal extension and verify $\knew$.

\begin{prop}
For any $A\in\cU$ there is a map $A\map TA$ in $\cL$ as demanded in $\knew$ and there is a universal extension $0\map J(A)\map T(A)\map A\map 0$ satisfying the requirements of axiom $\ke$.
\end{prop}

\begin{proof}
Let us first suppose that $A$ is a separable $C^*$-algebra. Let $\cpc(A)$ denote a category, whose objects are completely positive contractive linear maps $A\map B$, where $B$ is a separable $C^*$-algebra. A morphism $(A\map B)\map (A\map C)$ is a $*$-homomorphisms $B\map C$ such that the following diagram commutes:

\beqn
\xymatrix{
A\ar[r]\ar[dr] & B\ar[d] \\
& C
}
\eeqn This category is non-empty as $A\overset{\id}{\map} A$ belongs to it. It also has an initial object, whose construction is explained after Definition 8.25. in \cite{CunMeyRos}. We call this initial object $s: A\map T(A)$. Note that in ibid. it is constructed in the category of all $C^*$-algebras and it is denoted by $\sigma_A: A\map T_\cpc A$. However, if $A$ is separable, then so is $T(A)$.

Now let $A$ be any separable $\sC$-algebra and, as usual, write $A\cong \ilim_n A_n$ with surjective connecting homomorphisms $\theta_n: A_n\map A_{n-1}$. For each separable $C^*$-algebra $A_n$ there is a universal extension $0\map J(A_n)\map T(A_n)\map A_n\map 0$, with a canonical completely positive contractive linear splitting $s_n: A_n\map T(A_n)$. The completely positive contractive linear map $s_{n-1}\circ\theta_n : A_n\map T(A_{n-1})$ induces a $*$-homomorphism $\tau_n: T(A_n)\map T(A_{n-1})$ such that the following diagram commutes

\beqn
\xymatrix{
T(A_n) \ar[d]_{\tau_n}\ar[r]^{\pi_n} & A_n\ar[d]^{\theta_n} \ar@/_2pc/[l]_{s_n}\ar[d] \\
T(A_{n-1}) \ar[r]^{\pi_{n-1}}& A_{n-1}\ar@/_2pc/[l]_{s_{n-1}}.
}
\eeqn which gives rise to a morphism of extensions

\beqn
\xymatrix{
0\ar[r] & J(A_n)\ar[r]\ar[d] & T(A_n)\ar[r]\ar[d]^{\tau_n} & A_n\ar[d]^{\theta_n} \ar[r] & 0\\
0\ar[r] & J(A_{n-1}) \ar[r] & T(A_{n-1})\ar[r] & A_{n-1}\ar[r] & 0.
}\eeqn The inverse limit of these extensions produces an extension $$0\map J(A)=\ilim_n J(A_n) \map T(A)=\ilim_n T(A_n)\map A\map 0,$$ which we claim is a universal extension. The completely positive contractive linear splitting is given by the map $s=\ilim_n s_n: A\cong \ilim_n A_n\map\ilim_n T(A_n)= T(A)$. Let $\alpha: A\map B\cong\ilim_n B_n$ be any completely positive contractive linear map between separable $\sC$-algebras. Using Satz 5.3.6 of \cite{Bonkat} we write $\alpha$ as a morphism of inverse systems $\{\alpha_n: A_n\map B_n\}$, where each $\alpha_n$ is a completely positive contractive linear map between separable $C^*$-algebras. They induce a morphism of inverse systems $\{T(\alpha_n):T(A_n)\map B_n\}$, which produces the unique $*$-homomorphism $T(A)\cong\ilim_n T(A_n)\map\ilim_n B_n\cong B$ with the desired properties. This verifies $\knew$.

Given any extension $0\map I\map E\map A\map 0$ of $\sC$-algebras, admitting a completely positive contractive linear splitting, we write it as an inverse limit of extensions of $C^*$-algebras $0\map I_n\map E_n\map A_n\map 0$ admitting a completely positive linear splitting for each $n$. Note that if there is a completely positive linear splitting of a surjective $*$-homomorphism between separable $C^*$-algebras, then there is also a completely positive and contractive linear splitting (see Remark 2.5. of \cite{CunSka}). Since, for each $n$, $0\map J(A_n)\map T(A_n)\map A_n\map 0$ is a universal extension of $A_n$ there is a morphism of extensions

\beqn
\xymatrix{
0\ar[r] & J(A_n)\ar[r]\ar[d] & T(A_n)\ar[r]\ar[d] & A_n\ar[d]^{\id} \ar[r] & 0\\
0\ar[r] & I_n \ar[r] & E_n\ar[r] & A_{n}\ar[r] & 0
}\eeqn which gives rise to a morphism of inverse systems of extensions. Consequently, there is a morphism between their inverse limits

\beqn
\xymatrix{
0\ar[r] & J(A)\ar[r]\ar[d]^{\epsilon_A} & T(A)\ar[r]\ar[d] & A\ar[d]^{\id} \ar[r] & 0\\
0\ar[r] & I \ar[r] & E\ar[r] & A\ar[r] & 0
}\eeqn The classifying map is uniquely determined by the choice of a completely positive contractive linear map $s: A\map E$. If $s,s'$ are two such linear maps then the linear homotopy $ts +(1-t)s'$ induces a homotopy between their corresponding classifying maps, making the classifying map unique up to a homotopy.
\end{proof}

\begin{defn} \label{newKK}
We define the bivariant $\K$-theory groups on the category of separable $\sC$-algebras as in Equation \eqref{kk}, i.e., $\skk_*(A,B):=\kk_*(A,B)$, where $\cU$ is the category of separable $\sC$-algebras and $\fC$ consists of all separable $C^*$-algebra extensions, admitting a completely positive linear splitting.
\end{defn}

\begin{rem} \label{sepKK}
Restricted to the category of separable $C^*$-algebras, these groups agree with Kasparov's $\KK$-groups. By construction, this theory enjoys the universal property that is described in Theorem \ref{UP}.
\end{rem}

\begin{rem}
The bivariant $\skk$-theory for separable $\sC$-algebras could have also been constructed by a general machinery of localization in triangulated categories. However, such a construction would result in a rather complicated description of the elements of the $\skk$-groups; whereas, in the above approach they are represented as $*$-homomorphisms between certain $\sC$-algebras. 
\end{rem}

\noindent
A $\sC$-algebra $A\cong\ilim_{n\in\NN} A_n$ is called {\em nuclear} if each $A_n$ is a nuclear $C^*$-algebra. Obviously, a commutative $\sC$-algebra is nuclear.

\begin{ex}
Let $X=\dlim_n X_n$ be a countably compactly generated space and $P$ be a principal $PU$-bundle on $X$. Then the $\sC$-algebra $\CT(X,P)\cong\ilim_n \CT(X_n,P_n)$ constructed in subsection \ref{CT} is nuclear. Indeed, each $C^*$-algebra $\CT(X_n,P_n)$ is type $I$ and hence nuclear by a well-known Theorem of Takesaki. They are also clearly separable.
\end{ex}

\begin{lem}
Let $0\map I\map A\map B\map 0$ be an exact sequence of $\sC$-algebras. Then $A$ is nuclear if and only if both $I$ and $B$ are nuclear. In other words, the category of nuclear $\sC$-algebras is closed under the passage to closed two-sided ideals, quotients (by such ideals), and extensions.
\end{lem}

\begin{proof}
The result follows easily from the corresponding result about $C^*$-algebras \cite{ChoEff1} and Proposition 5.3 (2) of \cite{NCP1}.
\end{proof}

\begin{defn}
For any separable $\sC$-algebra $A$, we define the (analytic) $\K$-homology of $A$ as $\K^*(A)=\skk_*(A,\CC)$.
\end{defn}

\begin{rem}
A. Bonkat introduced a bivariant $\K$-theory $\BKK$ for inverse systems of $C^*$-algebras \cite{Bonkat}, which is also applicable to $\sC$-algebras. There is yet another bivariant $\K$-theory for pro $C^*$-algebras defined by Weidner \cite{Weidner}. It follows from Satz 5.3.11 of \cite{Bonkat} that the two bivariant theories are naturally isomorphic on the category of separable and nuclear $\sC$-algebras.
\end{rem}

\begin{prop}
The bivariant $\K$-theory groups $\skk$ described above are naturally isomorphic to the bivariant $\K$-theory groups $\BKK$ defined by Bonkat on the category of separable $\sC$-algebras.
\end{prop}

\begin{proof}
The assertion follows from Satz 5.3.10 of \cite{Bonkat}, which characterizes $\BKK$ on the category of separable $\sC$-algebras as the universal additive category valued functor with the properties $(E1)$, $(E2)$ and $(E3')$ as in Theorem \ref{UP} and the Remark thereafter.
\end{proof}

\begin{cor}
On the category of separable and nuclear $\sC$-algebras both $\skk$-theory and $\BKK$-theory agree naturally with Weidner's bivariant $\K$-theory.
\end{cor}

\noindent
This enables us to deduce some properties, which are very useful for computation.

\begin{cor} \label{Nuc}
Let $A=\ilim_n A_n$ be any nuclear and separable $\sC$-algebra. Then the functor $\K^*(A)=\skk_*(A,\CC)$ satisfies the following properties:

\begin{enumerate}

\item $\K^*(A(0,1)^2)\cong\K^*(A)$ [Bott periodicity]

\item $\K^*(A)\cong\dlim_n \K^*(A_n)$ [contravariant continuity]

\item Let $\{A_n\}$ be a countable family of nuclear and separable $C^*$-algebras. Then $$\K^*({\prod}_n A_n)\cong\oplus_n \K^*(A_n).$$
\end{enumerate}
\end{cor}

\begin{proof}
In view of the above Corollary, all the assertions follow from the corresponding results in Weidner's bivariant $\K$-theory \cite{Weidner}.
\end{proof}

\begin{rem}
For any separable $\sC$-algebra $A$, one can also define a new $\K$-theory as $\K^{\textup{new}}_*(A)=\skk_*(\CC,A)$. However, this $\K^{\textup{new}}$-theory will be naturally isomorphic to $\RK$-theory or the $\K$-theory of locally convex algebras defined above (when restricted to separable $\sC$-algebras).
\end{rem}

\begin{ex}
Let $P$ be a principal $PU$-bundle on $SU(\infty)$, such that its cohomology class is $\ell\in\H^3(SU(\infty),\ZZ)\simeq\ZZ$ as in Example \ref{SUK}. We know that the twisted $\K$-homology groups of the pair $(SU(n),\ell)$ are described by
$$\K^\bullet(\CT(SU(n),\iota^*_n(P)))\cong\K_\bullet(SU(n),\ell)\simeq  (\ZZ/c(n,\ell))^{2^{n-1}},$$ where $\K_\bullet=\K_0\oplus\K_1$ and $c(n,\ell)=gcd\{\binom{\ell + i}{i}-1\,:\, 1\leq i\leq n-1\}$. By the above Corollary, one deduces that $$\K_\bullet(SU(\infty),\ell)\cong \dlim_n\K_\bullet(SU(n),\ell)\simeq \dlim_n (\ZZ/c(n,\ell))^{2^{n-1}}.$$ This is again the trivial group.
\end{ex}

\section{Twisted local cyclic homology via separable $\sC$-algebras} \label{HL}
The natural target for a Chern--Connes character type map from $\K$-homology is some version of cyclic cohomology. However, neither periodic nor entire cyclic cohomology produces satisfactory results for $C^*$-algebras. For a separable and nuclear $C^*$-algebra $A$, the entire and the periodic cyclic cohomologies agree (see \cite{Khal}) and one finds $\HP^0(A)\cong\HH^0(A)$, $\HP^1(A)\cong\{0\}$. This result in not very satisfactory from the geometric viewpoint. Therefore, in this section we extend the theory of bivariant local cyclic homology to separable $\sC$-algebras so that we can construct a bivariant Chern-Connes type character taking values in it. Local cyclic (co)homology theory was developed by Puschnigg \cite{Puschnigg,PuschniggHL} and in its general setup it works on a nice category of "ind-Banach algebras". In order to apply this theory we spend some time in constructing such an ind-Banach algebra from a $\sC$-algebra in a functorial manner. Construction actually `factors through the world of 
bornological algebras'. We get a streamlined approach towards local cyclic (co)homology by using Meyer's presentation of the topic in the language of bornological and ind-Banach algebras. In some categorical aspects bornological vector spaces behave somewhat better than topological vector spaces, but we do not dwell on this point here. Since the theory of bornological algebras may not be a part of the standard toolkit of an operator algebraist or a geometer, we include a brief review of some its salient features here. There are a few contemporary textbooks explaining the theory more comprehensively, for instance, \cite{MeyCycHom,CunMeyRos}. We also refer the readers to Waelbroeck's survey \cite{Waelbroeck}, who has done extensive work in developing the theory, and the book of Hogbe-Nlend \cite{Nlend}. In what follows all bornological vector spaces and algebras are tacitly assumed to be complete, which simplifies the discussion.

\subsection{From a $\sC$-algebra to an ind-Banach algebra}
Intuitively, a bornological algebra is an algebra with a prescribed collection of bounded subsets (as opposed to open subsets, which would make it a topological algebra). More specifically, one calls a subset $B\subset V$, where $V$ is a $\CC$-linear space, a {\em disk} if it satisfies:

\begin{itemize}
\item $tx +(1-t)y\in B$ for all $x,y\in B$ and $t\in[0,1]$,

\item $\lambda B=\{\lambda b\,|\, b\in B\}\subset B$ for all $\lambda\in\CC$ with $|\lambda |\leqslant 1$,

\item $B=\cap_{\epsilon>0} \{(1+\epsilon)b\,|\, b\in B\}$.
\end{itemize} If $B\subset V$ is a disk, then the span $V_B=\RR_+ B=\{rb\,|\, r\in\RR_+, b\in B\}$ becomes a seminormed $\CC$-linear space, via the seminorm $\nu_B (v)=\textup{inf}\{r\in\RR_+\,|\, v\in rB\}$. If $V_B$ is actually a Banach space, then $B$ is said to be {\em complete}.

\begin{defn}
A (complete convex) {\em bornological $\CC$-vector space} is a $\CC$-linear space $V$ endowed with a family $\cS$ of subsets of $V$ satisfying the following axioms:

\begin{enumerate}
\item $S_1\in\cS$ and $S_2\subset S_1$, then $S_2\in\cS$,

\item if $S_1,S_2\in\cS$, then $S_1\cup S_2\in\cS$,

\item $\{v\}\in\cS$ for all $v\in V$,

\item if $S\in\cS$ and $c\in\RR_+$, then $cS\in \cS$,

\item any $S\in\cS$ is contained in some $B\in \cS$, where $B$ is a complete disk. \label{CompConv}
\end{enumerate} The family $\cS$ is called a {\em bornology} on $V$ and the subsets in $\cS$ are called the {\em bounded subsets} of $V$.
\end{defn}

\begin{rem}
More general definitions of bornologies exist in the literature. What we have defined above is a {\em complete convex} bornology.  The axiom \ref{CompConv} above imposes these conditions. Since we are not going to discuss more general bornologies, we drop the adjectives altogether.
\end{rem}

\begin{ex} \label{cborn}
Let $V$ be a Fr{\'e}chet space, i.e., a complete, metrizable, locally convex space. The family $\cS=\{W\subset V\,|\, \text{$\overline{W}$ is compact}\}$ defines a bornology on $V$. It is called the {\em precompact bornology} on $V$ and its subsets are called {\em precompact subsets}. For any Fr{\'e}chet space $V$, we denote the bornological vector space $V$, endowed with the precompact bornology by $\pct(V)$.
\end{ex}

\noindent
A $\CC$-linear map $f:(V_1,\cS_1)\map (V_2,\cS_2)$ between bornological vector spaces is called {\em bounded} if $S\in\cS_1$ implies $f(S)\in\cS_2$. For any two bornological vector spaces $V_1,V_2$, one defines the product bornological vector space as the $\CC$-linear space $V_1\times V_2$ equipped with the coarsest bornology making both projection maps $V_1\times V_2\map V_i$ with $i=1,2$ bounded. The bornological vector space $V_1\times V_2$ is complete if so are both $V_1$ and $V_2$.

The canonical inclusion of the category of all complete bornological vector spaces inside that of all (not necessarily complete) bornological vector spaces admits a left adjoint, which is called the {\em completion functor}. Being a left adjoint it commutes with all inductive limits. Since all our bornological vector spaces are assumed to be complete, one needs to apply this functor tacitly, whenever one runs into an incomplete one. The {\em complete bornological tensor product} $\brot$ between two bornological vector spaces $V,W$ is defined by the universal property: there is a canonical bounded bilinear map $\pi: V\times W\map V\brot W$, such that given any bounded bilinear map $\theta:V\times W\map Z$ into a bornological vector space $Z$, there is a unique bounded bilinear map $\theta':V\brot W\map Z$ satisfying $\theta=\theta'\circ\pi$. We are going to describe the explicit construction of the completed tensor product after introducing the dissection functor (see Equation \eqref{brot}).

A {\em bornological algebra} $A$ is a bornological vector space endowed with an associative, bilinear, and bounded multiplication map. Hence the multiplication map induces a bounded linear map $A\brot A\map A$. Now we establish a connection between Fr{\'e}chet algebras and bornological algebras.

\begin{lem}
The association $V\mapsto \pct(V)$, which is identity on morphisms, defines a fully faithful functor from the category of Fr{\'e}chet spaces with continuous linear maps to that of bornological vector spaces and bounded linear maps.
\end{lem}

\begin{proof}
One needs to observe that any continuous homomorphism is also bounded, i.e., it preserves precompact subsets. This says that the functor is faithful. That it is full follows from Theorem 1.29 of \cite{MeyCycHom}.
\end{proof}

\begin{cor}
If $A$ is Fr{\'e}chet algebra, then $\pct(A)$ is a bornological algebra and the association is functorial.
\end{cor}

\begin{proof}
This follows from the fact $\pct(V)\brot\pct(W)\cong \pct(V\grot W)$, where $V, W$ are Fr{\'e}chet spaces (see Theorem 1.87 of \cite{MeyCycHom}.
\end{proof}

The procedure of dissection enables us to move from the bornological world to the topological world. Let $V$ be a bornological vector space and let $\cS_c(V)$ be the set of complete bounded disks in $V$. It is a directed set under the relation $B_1\leq B_2$ if there exists a $c\in\RR_+$, such that $B_1\subset cB_2$. Any $B\in\cS_c(V)$ determines a Banach subspace $V_B\subset V$ and if $B_1\leq B_2$, then there is an injective bounded linear map $V_{B_1}\map V_{B_2}$. This procedure produces an inductive (directed) system of Banach spaces indexed by $\cS_c(V)$. For any bornological vector space $V$ the inductive systems of Banach spaces indexed by $\cS_c(V)$ thus obtained is denoted by $\diss(V)$. This construction is actually functorial with respect to bounded linear maps, i.e., any bounded linear map between bornological vector spaces $V\map W$ induces a morphism of inductive systems of Banach spaces $\diss(V)\map\diss(W)$. The functor from bornological vector spaces to inductive systems of Banach spaces 
admits a left adjoint, which is $\sep\dlim$. The bornological tensor product, which was defined by its universal property above, can be described in terms of these functor as:

\beq \label{brot}
V\brot W\cong \sep\dlim_{(B,B')} V_B\grot W_{B'}, \;\;\; B\in\cS_c(V), B'\in\cS_c(W),
\eeq where $\grot$ denotes Grothendieck's projective tensor product between Banach spaces.

Given any category $\cC$, one can construct its ind-category $\overset{\map}{\cC}$. We do not belabour the concept of ind-categories. Let us simply mention that its objects are formal diagrams $F_I:I\map\cC$, where $I$ is a small filtering category and, by definition, $$\Hom_{\overset{\map}{\cC}}(F_I,F'_J)=\ilim_i \dlim_j \Hom_\cC(F_I(i),F'_J(j)), \;\; i\in I, j\in J.$$
Let $\Ban$ denote the symmetric monoidal category of Banach spaces with $\grot$ giving the symmetric monoidal structure. Then $\indBan$ denotes the symmetric monoidal category of inductive systems of Banach spaces with the monoidal structure given by $\grot$ (extended naturally to inductive systems $A_I \grot B_J:= \{A_i\grot B_j\}_{(i,j)\in I\times J}$). For simplicity we continue to denote the monoidal structure on $\indBan$ by $\grot$. The constant system $\CC$ is the unit object of $\indBan$. Given any symmetric monoidal category, one can talk about a monoid object in that category. The ind-Banach algebras are precisely the monoid objects in $\indBan$. It should be noted that an ind-Banach algebra is not necessarily an inductive system of Banach algebras; the ones that are inductive systems of Banach algebras are called {\em locally multiplicative}.

\begin{rem}
Observe that a bornological algebra can be viewed as a monoid object in the symmetric monoidal category of bornological vector spaces equipped with $\brot$. The dissection functor $\diss$ is unfortunately not symmetric monoidal in general, i.e., $\diss(V\brot W)\ncong\diss(V)\grot\diss(W)$. However, if $V,W$ are Fr{\'e}chet spaces, then there is a natural isomorphism $\diss\circ\pct(V)\grot\diss\circ\pct(W)\cong \diss\circ\pct(V\grot W)$ (see Theorem 1.166. of \cite{MeyCycHom}). Therefore, the functor $\diss\circ\pct(-)$ preserves monoid objects between Fr{\'e}chet spaces and ind-Banach spaces, i.e., it sends a Fr{\'e}chet algebra to an ind-Banach algebra.
\end{rem}

Given any Fr{\'e}chet algebra $A$ (in particular, a $\sC$-algebra), by applying the composite functor $P(A):=\diss\circ\pct(A)$, we obtain an ind-Banach algebra, which establishes the functorial passage from Fr{\'e}chet algebras (in particular, $\sC$-algebras) to ind-Banach algebras alluded to above. Thus Meyer's technology may be deployed to define the (bivariant) local cyclic homology of $P(A)$.

\begin{rem}
Note that, whilst a $\sC$-algebra $A$ is defined as a countable inverse limit of $C^*$-algebras, the associated ind-Banach algebra $P(A)$ is merely expressed as a monoid object in the category of inductive systems of Banach spaces. These algebras are not {\em locally multiplicative} in the sense of section 3.2.2. of \cite{MeyCycHom} in general.
\end{rem}

\subsection{Local cyclic homology via the analytic tensor algebra} \label{AnaTen}
Recall that the periodic cyclic homology was defined in Section \ref{HPT} using the $\X$-complex of a quasi-free (tensor) algebra. We adopt a similar approach for defining local cyclic (co)homology. For any ind-Banach algebra $A$, construct the algebra of differential forms (resp. even differential forms) with Fedosov product $(\Omega_\alg(A),\circ)$ (resp. $(\Omega^\ev_\alg(A),\circ)$) purely algebraically, i.e., construct it formally in the symmetric monoidal category $(\indBan, \grot)$ using the fact that it is a monoid object. Write $A\cong \{ A_i\}_{i\in I}$, with $I$ directed and each $A_i$ a Banach space. Equip each $A_i$ with a closed unit ball $B_i$ and define $B_{i,n}=nB_i$ for all $(i,n)\in I\times\NN$. Set $$\langle\langle B_{i,n}\rangle\rangle =B_{i,n}\cup B_{i,n}(dB_{i,n})^\infty\cup (dB_{i,n})^\infty,$$ where $d$ is the formal differential in $\Omega_\alg(A)$ and $(dB_{i,n})^\infty=\cup_{k=1}^\infty (dB_{i,n})^k$. Let $\overline{\langle\langle B_{i,n}\rangle\rangle}$ denote the minimal 
complete bounded bounded disk in $A_i$ containing $\langle\langle B_{i,n}\rangle\rangle$. We denote the completion of $\Omega_\alg(A_i)$ (resp. $\Omega^\ev_\alg(A_i)$ with respect to the norm defined by $\overline{\langle\langle B_{i,n}\rangle\rangle}$ by $\Omega_\an(A_{i,n})$ (resp. $\Omega^\ev_\an(A_{i,n})$. Letting $I\times\NN$ be directed in the obvious manner, we get an inductive system of Banach spaces $\{\Omega_\an(A_{i,n})\}_{I\times\NN}$. The Fedosov product and the differential extend to this inductive system, making it an ind-Banach algebra. The ind-Banach subalgebra $\{(\Omega^\ev_\an(A_{i,n})\}$ is by definition the {\em analytic tensor algebra} $\cT_\an(A)$. One can now construct the $\X$-complex of $\cT_\an(A)$. We omit the details of this construction, which is very similar to the one described in Section \ref{HPT}. The interested readers can also find the details in Section 5.2 of \cite{MeyCycHom}.

\begin{rem}
The functor $A\mapsto P(A)$ converts a Fr{\'e}chet algebra $A$ into an ind-Banach algebra. We simplify notations by dropping $P$. It is tacitly assumed that the machinery discussed below is applied to a Fr{\'e}chet algebra after converting it to an ind-Banach algebra by applying the functor $P$.
\end{rem}

\subsection{Local homotopy category of $\ZZ/2$-graded complexes} \label{locHom} Let $\cC$ denote an additive category. One can form the triangulated homotopy category of $\ZZ/2$-graded chain complexes in $\cC$ with the mapping cone triangles as the prototypical exact triangles. We denote this category by $\Ho\cC_\bullet$. Recall that a $\ZZ/2$-graded complex is one of the form

\beqn
\cdots \overset{d_0}{\map} C_1\overset{d_1}{\map} C_0 \overset{d_0}{\map} C_1\overset{d_1}{\map} C_0 \overset{d_0}{\map} \cdots
\eeqn and the morphisms in the homotopy category are the homotopy classes of chain maps. In general, for any two chain complexes $X_\bullet,Y_\bullet$, there is a mapping chain complex $\Hom(X_\bullet,Y_\bullet)$, whose $n$-cycles are maps $X_\bullet[n]\map Y_\bullet$ of graded objects, and the differential being the graded commutator $d(f)=f d_Y - (-1)^{|f|} d_X f$. It follows that the morphisms between $X_\bullet$ and $Y_\bullet$ in the homotopy category $\Ho\cC_\bullet$ is given by $\H_0(\Hom(X_\bullet,Y_\bullet))$.

Now set $\cC=\overset{\longrightarrow}{\Ban}$ to be the additive category of inductive systems of Banach spaces. For any $X_\bullet, Y_\bullet\in\Ho\cC_\bullet$ one can define the functorial mapping complex $\Hom(X_\bullet,Y_\bullet)$. The {\em local homology groups} of a specific mapping complex will eventually compute the bivariant local cyclic homology of a pair of ind-Banach algebras. As the name suggests, the local homology of an object in $\Ho\cC_\bullet$ is not the same as its na{\"{i}}ve homology. It is obtained by passing to a localization of the triangulated category $\Ho\cC_\bullet$ and then taking its homology. Any Banach space can be viewed as an inductive system of Banach spaces via the constant system. Similarly, any chain complex of Banach spaces can be viewed as a {\em finitely presented} object of $\Ho\cC_\bullet$. We call a chain complex $Y_\bullet$ in $\Ho\cC_\bullet$ {\em locally contractible} if for any chain complex of Banach spaces $X_\bullet$, one has $\H_*(\Hom(X_\bullet,Y_\bullet))
=\{0\}$. A chain map $f:Y_\bullet\map Y'_\bullet$ is called a {\em local homotopy equivalence} if and only if the mapping cone of $f$ is locally contractible. An exact functor from $\Ho\cC_\bullet$ to any other triangulated category is called {\em local} if it sends a local homotopy equivalence to an isomorphism. The {\em local homotopy category of $\ZZ/2$-graded chain complexes}, denoted by $\Ho\cC_\bullet^\loc$, is by definition the codomain of the universal local functor $\loc:\Ho\cC_\bullet\map\Ho\cC_\bullet^\loc$. Formal localization theory of triangulated categories ensures its existence. For any $X_\bullet,Y_\bullet\in\Ho\cC_\bullet$, we define the {\em bivariant local homology} as $\H_n^\loc(X_\bullet,Y_\bullet)=\Hom_{\Ho\cC_\bullet^\loc}(X_\bullet,Y_\bullet[n])$. For any ind-Banach algebra $A$, it is clear that the $\X$-complex $\X(\cT_\an(A))\in\Ho\cC_\bullet$. It turns out that the analytic tensor algebra $(\Omega^\ev_\an(A),\circ)$ constructed above is actually analytically quasi-free, so that 
one may define the (bivariant) local cyclic homology groups using the $\X$-complex (see Theorem 5.38 of ibid.).

\begin{defn} \label{HLdefn}
Given any two ind-Banach algebras $A,B$, one defines the bivariant local cyclic homology groups as $\HL_*(A,B)=\H_*^\loc(\X(\cT_\an(A)),\X(\cT_\an(B)))$.

\noindent
In particular, $\HL_0(A,B)=\Hom_{\Ho\cC_\bullet^\loc}(\X(\cT_\an(A)),\X(\cT_\an(B)))$, where $\cC=\overset{\longrightarrow}{\Ban}$, and the local cyclic homology (resp. cohomology) groups of $A$ are defined as $$\text{$\HL_*(A)=\HL_*(\CC,A)$ (resp. $\HL^*(A)=\HL_*(A,\CC)$).}$$

\noindent
If $A,B$ are Fr{\'e}chet algebras, then by definition $\HL_*(A,B)=\HL_*(P(A),P(B))$.
\end{defn}

\section{The bivariant Chern--Connes type character} \label{HLCC}
Now we construct a bivariant Chern--Connes type character from the bivariant $\K$-theory to the bivariant local cyclic homology for separable $\sC$-algebras and then specialize to the dual Chern--Connes character from (analytic) $\K$-homology to local cyclic cohomology.

\begin{lem} \label{stability}
If a functor $F$ from the category of separable $\sC$-algebras is homotopy invariant, then $F(\cpt\prot -)$ is $C^*$-stable.
\end{lem}

\begin{proof}
Since the corner embedding $\iota:\cpt\map\cpt\prot\cpt$ is homotopic to an isomorphism (say $\alpha$), there is commutative diagram of $*$-homomorphisms

\beqn
 \xymatrix{
 & \cpt\prot\cpt \\
\cpt
\ar[ur]^{\iota}
\ar[rr]^{h}
\ar[dr]_{\alpha}
&& \cpt\prot\cpt\prot\C([0,1])
\ar[ul]_{\ev_0}
\ar[dl]^{\ev_1} \\
& \cpt\prot\cpt, } \\
\eeqn Now applying $-\prot A$ to the above diagram one obtains the homotopy between the corner embedding $\iota\prot\id:\cpt\prot A\map \cpt\prot\cpt\prot A$ and an isomorphism $\alpha\prot\id$. Finally, by the homotopy invariance of $F$, one concludes that $F(\iota\prot\id)=F(\alpha\prot\id)$ is an isomorphism.
\end{proof}

\noindent
The composition of morphisms in $\Ho\cC_\bullet^\loc$ induces a natural associative composition product $$\HL_*(A,B)\otimes\HL_*(B,C)\map\HL_*(A,C)$$ for any three ind-Banach algebras (see Proposition 5.8 of \cite{PuschniggHL}). This enables us to define an additive category, whose objects are $\sC$-algebras and the morphisms are $\HL_0$-groups with the composition of morphisms given by the above product. Of course, one needs to apply the functor $P(-)$ to convert a $\sC$-algebra into an ind-Banach algebra but, as mentioned before, we suppress the application of $P(-)$ from our notations below. We denote the category, whose objects are separable $\sC$-algebras and the morphisms are the bivariant local cyclic homology $\HL_0$-groups, by $\HLcat$. Any $*$-homomorphism $A\map B$ between $\sC$-algebras induces a map of $\X$-complexes $\X(\cT_\an(A))\map\X(\cT_\an(B))$, which eventually induces a morphism in $\Ho\cC_\bullet^\loc$ giving rise to an element in $\HL_0(A,B)$. The composition of $*$-homomorphisms is 
compatible with the above-mentioned product in bivariant local cyclic homology, whence there is a canonical functor $\sCalg\map\HLcat$ that applies $P(-)$ to the objects. Recall that $\sCalg$ denotes the category of separable $\sC$-algebras with $*$-homomorphisms.

\begin{prop} \label{HLprop}
The covariant functor $\sCalg\functor\HLcat$, sending $A\mapsto \cpt\prot A$, has the property that the associated functors $\Hom_{\HLcat}(\cpt\prot A,\cpt\prot -) $ and $\Hom_{\HLcat}(\cpt\prot -,\cpt\prot A)$ satisfy the properties $(E1)$, $(E2)$ and $(E3)$ for all $A\in\sCalg$ as in Theorem \ref{UP}.
\end{prop}

\begin{proof}
Recall that $\Hom_{\HLcat}(A,-)=\HL_0(A,-)$ and $\Hom_{\HLcat}(-,A)=\HL_0(-,A)$. The property $(E1)$ follows from Corollary 7 of \cite{HLHomotopy}. The functor $P(\cpt\prot -)$ sends a semi-split extension of $\sC$-algebras to such an extension of ind-Banach algebras. It follows from the excision results of local cyclic (co)homology with respect to such extensions (see Theorem 5.13 of \cite{PuschniggHL} and Theorem 5.77 of \cite{MeyCycHom}) that property $(E3')$ as in Remark \ref{E3form} is satisfied, which is sufficient for our purposes. The functor $\HL_*(\cpt\prot -,D)$ is $C^*$-stable by the above Lemma \ref{stability}, since the functor $\HL_*(\cpt\prot-,D)$ is homotopy invariant for any fixed separable $\sC$-algebra $D$. A similar argument shows that $\HL_*(D,\cpt\prot -)$ is also $C^*$-stable thus verifying $(E2)$.
\end{proof}

\begin{thm} \label{BivCC}
There is a natural multiplicative bivariant Chern--Connes type character $\ch^\biv_*:\skk_*(A,B)\cong\skk_*(\cpt\prot A,\cpt\prot B)\map\HL_*(\cpt\prot A,\cpt\prot B)$.
\end{thm}

\begin{proof}
The assertion follows immediately from the Theorem \ref{UP}, Remark \ref{sepKK}, and the above Proposition \ref{HLprop} and standard arguments following \cite{CunBivCh}.
\end{proof}

\begin{rem}
The bivariant Chern--Connes type character $\ch^\biv_*$ constructed above obviously differs from the one described in Section \ref{HPCC} (also denoted by $\ch^\biv_*$), which was $\HP_*$-valued, on the category of $m$-algebras. Restricted to the category of separable $C^*$-algebras, our character $\ch^\biv_*$ agrees with Puschnigg's $\ch^P_*$ up to a natural isomorphism, that we encountered in subsection \ref{Pusch}. This follows from the unique characterization of the bivariant Chern--Connes character on the category of separable $C^*$-algebras (see Theorem \ref{PCC}) and the $C^*$-stability of bivariant $\HL_*$ on the category of separable $C^*$-algebras (see Theorem 5.14 (c) of ibid.), which provides the natural identification $\HL_*(\cpt\prot A,\cpt\prot B)\cong\HL_*(A,B)$. The following commutative diagram illustrates the situation:

\beqn
\xymatrix{
&& \HL_*(\cpt\prot A,\cpt\prot B)\ar[d]^\cong \\
\KK_*(A,B)\cong\kk_*(A,B) \ar[rru]^{\ch^\biv_*}\ar[rrd]_{\ch^P_*} && \HL_*(A,\cpt\prot B) \\
&& \HL_*(A,B)\ar[u]_\cong.
}
\eeqn
\end{rem}

\begin{thm} \label{dualCC}
Let $A\cong\ilim_n A_n$ be a separable and nuclear $\sC$-algebra. Then there is a natural dual Chern--Connes character homomorphism $\ch^*(A):=\ch^\biv_*(A,\CC):\K^*(A)\map\HL^*(\cpt\prot A)$, which factorizes as

\beq \label{dualfac}
\xymatrix{
\K^*(A)\ar[rr]^{\ch^*(A)}\ar[dr] && \HL^*(\cpt\prot A)\\
&\dlim_n \HL^*(A_n) \ar[ur].
}
\eeq
\end{thm}

\begin{proof}
Putting $B=\CC$ in the above Theorem \ref{BivCC}, we get a natural homomorphism $\ch^*(A):\K^*(A)\map\HL_*(\cpt\prot A,\cpt)$. Since the corner embedding $A\map\cpt\prot A$ is an $\HL$-equivalence for any $C^*$-algebra $A$ (see Theorem 6.25 of \cite{MeyCycHom}), there is a natural isomorphism $\HL_*(\cpt\prot A,\cpt)\cong\HL_*(\cpt\prot A,\CC)=\HL^*(\cpt\prot A)$, giving rise to the natural homomorphism $\ch^*(A):\K^*(A)\map\HL^*(\cpt\prot A)$, which is the dual Chern--Connes character.

Using the projection homomorphisms $p_n:A\map A_n$, one obtains directed systems of abelian groups $\{\K^*(A_n),p_n^*\}$ and $\{\HL^*(\cpt\prot A_n),p_n^*\}$ with canonical homomorphisms $\dlim_n\K^*(A_n)\map \K^*(A)$ and $\dlim_n\HL^*(\cpt\prot A_n)\map\HL^*(\cpt\prot A)$. The naturality of the map $\ch^*$ produces the following commutative diagram

\beqn
\xymatrix{
\K^*(A)\ar[rr]^{\ch^*(A)} && \HL^*(\cpt\prot A)\\
\dlim_n\K^*(A_n)\ar[rr]^{\dlim_n \ch^*(A_n)}\ar[u]&&\dlim_n \HL^*(\cpt\prot A_n) \ar[u].
}
\eeqn Now the desired factorization is obtained by observing that $\dlim_n\HL^*(\cpt\prot A_n)\cong\dlim_n\HL^*(A_n)$ and that the left vertical arrow is an isomorphism, since $A$ is a nuclear and separable $\sC$-algebra (see Corollary \ref{Nuc}).

\end{proof}

\begin{rem}
If, in addition, $A$ is stable, i.e., $A\cong \cpt\prot A$, then obviously one can identify $\HL^*(\cpt\prot A)\cong\HL^*(A)$ in the above diagram \eqref{dualfac}. All the twisted continuous trace $\sC$-algebras $\CT(X,P)$, that we encountered before, are stable.
\end{rem}

\begin{prop} \label{lastprop}
Let $A\cong\ilim_n A_n$ be a separable and nuclear $\sC$-algebra, such that for all $n$ the dual Chern--Connes character $\K^*(A_n)\map\HL^*(\cpt\prot A_n)\cong\HL^*(A_n)$ is an isomorphism after tensoring with the complex numbers. Then the map induced by $\ch^*(A)$ (see Equation \eqref{dualfac}) $$\K^*(A)\cong\dlim_n \K^*(A_n)\map\dlim_n \HL^*(\cpt\prot A_n)\cong\dlim_n\HL^*(A_n)$$ is an isomorphism after tensoring with the complex numbers.
\end{prop}

\begin{proof}
One simply needs to observe that tensoring with the complex numbers commutes with inductive limits in abelian groups.
\end{proof}

\begin{rem}
For $\sC$-algebras of the form $\CT(X,P)=\ilim_n\CT(X_n,P_n)$, that we have seen before, the Chern--Connes character from $\K$-theory to local cyclic homology $$\K_*(\CT(X_n,P_n))\map\HL_*(\CT(X_n,P_n))$$ becomes an isomorphism after tensoring with the complex numbers for all $n$ (see Theorem \ref{HLcpt}). However, it does not automatically induce an isomorphism between their inverse limits, as tensoring with the complex numbers does not commute with inverse limits.
\end{rem}

\noindent
In the rest of the paper we discuss a rather general case, where the hypotheses of the above Proposition \ref{lastprop} are satisfied.

\begin{ex}
Recall that there is an associative cup-cap product in Kasparov's $\KK$-theory \cite{KasNovikov} given by

$$\KK_0(A_1,B_1\prot D)\otimes_D\KK_0(D\prot A_2,B_2)\map \KK_0(A_1\prot A_2, B_1\prot B_2),$$ which is functorial in each variable. Connes suggested a notion of Poincar{\'e} duality in noncommutative geometry using the above product \cite{ConBook}. If $A,B$ are two separable $C^*$-algebras, then $(A,B)$ is called a {\em Poincar{\'e} dual pair} (PD pair) if there are elements $\alpha\in\KK_0(A\prot B,\CC)$, $\beta\in\KK_0(\CC,A\prot B)$, such that $$\text{$\beta\otimes_A \alpha = 1_B\in\KK_0(B,B)$ and $\beta\otimes_B \alpha = 1_A\in\KK_0(A,A)$.}$$ They induces isomorphisms between $\K$-theory of $A$ and $\K$-homology of $B$ and vice versa as follows:

\beqn
\K_*(A)\cong\KK_*(\CC,A) &\map & \KK_*(B,\CC)\cong\K^*(B)\\
x &\mapsto & x\otimes_A \alpha,
\eeqn whose inverse is given by the map

\beqn
\K^*(B)\cong\KK_*(B,\CC) &\map & \KK_*(\CC,A)\cong\K_*(A)\\
y &\mapsto & \beta\otimes_B y.
\eeqn Thanks to the multiplicativity of the bivariant Chern--Connes character, the element $\ch^\biv_*(\alpha)$, $\ch^\biv_*(\beta)$ induce similar isomorphisms between the $\HL_*(A)$ and $\HL^*(B)$ and vice versa. If $(A,B)$ is a PD pair, then Puschnigg's bivariant Chern--Connes character induces a commutative diagram:

\beqn
\xymatrix{
\K^*(A)\ar[r]^\cong \ar[d]_{\ch^*} & \K_*(B)\ar[d]^{\ch_*}\\
\HL^*(A) \ar[r]^\cong & \HL_*(B).
}
\eeqn where the horizontal arrows are the Poincar{\'e} duality isomorphisms as described above. For some examples of physical PD pairs we refer the readers to \cite{MosPD,BMRS2}. Therefore, if $A\cong\ilim_n A_n$, $B\cong\ilim_n B_n$ are separable and nuclear $\sC$-algebras, such that each $(A_n,B_n)$ is a PD pair and $\ch_*(B_n):\K_*(B_n)\map\HL_*(B_n)$ is an isomorphism after tensoring with the complex numbers for all $n$, then the hypotheses of Proposition \ref{lastprop} are satisfied. We know that if each $B_n$ belongs to the UCT class, then $\ch_*(B_n)$ is an isomorphism after tensoring with the complex numbers (see Theorem 7.7 of \cite{MeyCycHom}).
\end{ex}

\begin{rem}
 The task of constructing a bivariant Chern--Connes type character on the category of pro $C^*$-algebras is a bit tricky. One can construct a bivariant $\K$-theory with the desired properties but the extension of bivariant local cyclic homology to pro $C^*$-algebras is somewhat problematic. Even if one could define it using the $\X$-complex formalism, it may not satisfy continuous homotopy invariance. One possibility is to localize the triangulated category $\Ho(\cC^\loc_\bullet)$ (see subsection \ref{locHom}) further along the evaluation maps $A[0,1]\map A$ for every pro $C^*$-algebra $A$ in order to enforce continuous homotopy invariance. 
\end{rem}


\bibliographystyle{abbrv}
\bibliography{/home/tubai/Professional/math/MasterBib/bibliography}

\vspace{5mm}
\noindent

\end{document}